\documentclass[11pt, notitlepage]{article}
\usepackage{amssymb,amsmath,comment}
\catcode`\@=11 \@addtoreset{equation}{section}
\def\thesection{\arabic{section}}

\def\theequation{\thesection.\arabic{equation}}
\catcode`\@=12
\usepackage{colortbl}%
\usepackage{a4wide}

\newcommand{\ds} {\displaystyle}
\newcommand{\e}{\epsilon}

\newcommand{\al} {\alpha}
\newcommand{\ba} {\beta}

\newcommand{\Ga} {\Gamma}
\newcommand{\Om} {\Omega}
\newcommand{\ra} {\rightarrow}

\newcommand{\De} {\Delta}
\newcommand{\la} {\lambda}
\newcommand{\La} {\Lambda}
\newcommand{\noi} {\noindent}

\newcommand{\mb} {\mathbb}
\newcommand{\mc} {\mathcal}

\newcommand{\ld} {\langle}
\newcommand{\rd} {\rangle}
\newcommand{\io} {\int_{\Om}}

\newcommand{\h}{Y}

\newcommand{\aw}{K_{\la,\mu}(u,w)}
\newcommand{\bw}{b(x)u_{+}^{\al}w_{+}^{\ba} dx}
\newcommand{\ak}{K_{\la,\mu}(u_k,w_k) }
\newcommand{\bk}{b(x)(u_k)_{+}^{\al}(w_{k})_{+}^{\ba}}

\newcommand{\fg}{\la \int_{\Om} f(x)|u|^{1-q} dx+\mu \int_{\Om}g(x)|w|^{1-q} dx}

\newcommand{\uw}{\|(u,w)\|^2}

\newcommand{\uk}{\|(u_k,w_k)\|^2}
\newcommand{\n}{\mc N_{\la,\mu}}

\newcommand{\ee}{(u,w)}

\usepackage[all]{xy}
\catcode`\@=11
\def\theequation{\@arabic{\c@section}.\@arabic{\c@equation}}
\catcode`\@=12

\def\QED{\hfill {$\square$}\goodbreak \medskip}

\newtheorem{Theorem}{Theorem}[section]
\newtheorem{Lemma}[Theorem]{Lemma}

\newtheorem{Corollary}[Theorem]{Corollary}

\newtheorem{Definition}[Theorem]{Definition}

\begin{document}
{\vspace{0.01in}}

\title
{Multiplicity results of fractional-Laplace system with sign-changing and singular nonlinearity}

\author{
{\bf  Sarika Goyal\footnote{email: sarika1.iitd@gmail.com}} \\
{\small Department of Mathematics}, \\{\small Indian Institute of Technology Delhi}\\
{\small Hauz Khas}, {\small New Delhi-16, India.}\\
 }
\date{}

\maketitle

\begin{abstract}

In this article, we study the following fractional-Laplacian system with singular nonlinearity
 \begin{equation*}
 (P_{\la,\mu}) \left\{
\begin{array}{lr}
(-\De)^s u =  \la f(x) u^{-q}+  \frac{\al}{\al+\ba}b(x) u^{\al-1} w^\ba\; \text{in}\;
\Om \\
(-\De)^s w =  \mu g(x) w^{-q}+ \frac{\ba}{\al+\ba} b(x) u^{\al} w^{\ba-1}\; \text{in}\;
\Om \\
 \quad \quad u, w>0\;\text{in}\;\Om, \quad u = w = 0 \; \mbox{in}\; \mb R^n \setminus\Om,\\
\end{array}
\quad \right.
\end{equation*}
where $\Om$ is a bounded domain in $\mb R^n$ with smooth
boundary $\partial \Om$, $n>2s$, $s\in(0,1)$, $0<q<1$, $\al>1$, $\ba>1$ satisfy $2<\al+\ba< 2_{s}^*-1$ with $2_{s}^*=\frac{2n}{n-2s}$, the pair of parameters $(\la,\mu)\in \mb R^2\setminus\{(0,0)\}$. The weight functions $f,g: \Om\subset\mb R^n \ra \mb R$ such that $0<f$, $g\in L^{\frac{\al+\ba}{\al+\ba-1+q}}(\Om)$, and $b:\Om\subset\mb R^n \ra \mb R$ is a sign-changing function such that $b(x)\in L^{\infty}(\Om)$. Using variational methods, we show existence and
multiplicity of positive solutions of $(P_{\la,\mu})$ with respect to the pair of parameters $(\la,\mu)$.
\medskip

\noi \textbf{Key words:} Fractional Laplacian system, singular nonlinearity, sign-changing weight function, Variational methods.

\medskip

\noi \textit{2010 Mathematics Subject Classification:} 35A15, 35J75, 35R11.
\end{abstract}

\bigskip
\vfill\eject

\section{Introduction}

\setcounter{equation}{0} Let $s\in (0,1)$ and let $0\in\Om\subset \mb
R^n$ is a bounded domain with smooth boundary, $n>2s$. Then we
consider the following fractional system with singular nonlinearity:
 \begin{equation*}
 (P_{\la,\mu}) \left\{
\begin{array}{lr}
(-\De)^s u =  \la f(x) u^{-q}+  \frac{\al}{\al+\ba}b(x) u^{\al-1}w^\ba\; \text{in}\;
\Om \\
(-\De)^s w =  \mu g(x) w^{-q}+ \frac{\ba}{\al+\ba} b(x) u^{\al} w^{\ba-1}\; \text{in}\;
\Om \\
 \quad \quad u, w>0\;\text{in}\;\Om, \quad u = w = 0 \; \mbox{in}\; \mb R^n \setminus\Om.\\
\end{array}
\quad \right.
\end{equation*}
Here, $(-\De)^{s}$ is the fractional Laplacian operator
defined as
\begin{equation*}
(-\De)^{s} u(x)= -\frac{1}{2} \int_{\mb
R^n}\frac{u(x+y)+u(x-y)-2u(x)}{|y|^{n+2s}} dy \;\text{for all} \;
x\in \mb R^n.
\end{equation*}
We assume the following assumptions on $f$ and $g$:
 \begin{enumerate}
  \item[(a1)] $f,g:   \Om\subset\mb R^n \ra \mb R$ such that $0< f,$ $g\in L^{q^*}(\Om)$, where $q^*=\frac{\al+\ba}{\al+\ba-1+q}$.
   \item[(b1)] $b:\Om\subset\mb R^n \ra \mb R$ is a sign-changing function such that $b^{+}=\max\{f,0\}\not\equiv 0$ and $b(x)\in L^{\infty}(\Om)$.
\end{enumerate}
Also the pair of parameters $(\la,\mu)\in \mb R^2\setminus\{(0,0)\}$, $0<q <1$ and $\al>1$, $\ba>1$ satisfy $2<\al+\ba<2^{*}_{s}-1$, with $2^{*}_{s}= \frac{2n}{n-2s}$, known as fractional critical Sobolev exponent.

\noindent In this work, we prove the existence of multiple non-negative solutions for a system of fractional operator with
singular and sign changing nonlinearity by studying the nature of Nehari manifold with respect to the parameter $\lambda$ and $\mu$. These same result can be easily extended to $p-$fractional Laplacian operator $(-\De)_{p}^s$, defined as
\begin{align*}
  (-\De)_{p}^s =& 2 \lim_{\e\ra 0} \int_{\mb R^n \setminus B_{\e}(x)} \frac{|u(y)-u(x)|^{p-2}(u(y)-u(x))}{|x-y|^{n+ps}}dy .
\end{align*}
This definition is consistent, up to a normalization constant depending on $n,$ $s$, with linear Laplacian fractional $(-\De)^s$, for the case $p=2$.

\noindent The natural space to look for solutions of the problem $(P_{\lambda,\mu})$ is the product space
$W^{s,p}_{0}(\Omega)\times W^{s,p}_{0}(\Omega)$. In order to study $(P_{\lambda,\mu})$, it is important to encode the `boundary condition'
$u=v=0$ in $\mathbb{R}^n\setminus\Omega$ in the weak formulation. Servadei and Valdinoci in \cite{mp} have  introduced
the new function spaces  to study the variational functionals related to {the} fractional Laplacian by observing the interaction between $\Omega$ and
$\mathbb{R}^n\setminus\Omega$.

\noindent For $u=v$, $\alpha=\beta$, $\al+\ba=r$, $\lambda=\mu$ and $f=g$, the problem $(P_{\lambda,\mu})$ reduces  to the following fractional equation with singular nonlinearities
\begin{equation*}
 (P_{\la}) \left\{
\begin{array}{lr}
(-\De)_{p}^s =  f(x) w^{-q}+ \la b(x) w^r\; \text{in}\;
\Om , \quad \quad w>0\;\text{in}\;\Om, \quad w = 0 \; \mbox{in}\; \mb R^n \setminus\Om,\\
\end{array}
\quad \right.
\end{equation*}
In \cite{cv}, the author studied the existence and multiplicity of non-negative solutions to problem $(P_{\la})$ for sign changing and singular nonlinearity. In the scalar case the problems involving the fractional operator  with singular nonlinearity have been studied by many authors, see \cite{tu} and references therein.

\noi The fractional power of Laplacian is the infinitesimal generator of L\'{e}vy stable diffusion process and arise in anomalous diffusions in plasma, population dynamics, geophysical fluid dynamics, flames propagation, chemical reactions in liquids and American options in finance. For more details, one can see \cite{da, gm} and reference therein. Recently the fractional elliptic equation attracts a lot of interest in nonlinear analysis such as in \cite{cs, mp,var,bn1,bn}. Caffarelli and Silvestre \cite{cs} gave a new formulation of fractional Laplacian through Dirichlet-Neumann maps. This is commonly used in the literature since it allows us to write a nonlocal problem to a local problem which allow us to use the variational methods to study the existence and uniqueness.

\noi On the other hand, the fractional elliptic problem have been investigated by many authors, for example, \cite{mp, var} for subcritical case, \cite{bn1, bn} for critical case with polynomial type nonlinearities. Moreover, by Nehari manifold and fibering maps, the author obtained the existence of multiple solutions for fractional equations for critical \cite{zlj} and subcritical case \cite{ssa,ssi} and reference therein. In case of square root of Laplacian, existence and multiplicity results for sublinear and superlinear type of nonlinearity with sign-changing
weight functions is studied in \cite{xy}. In \cite{xy}, author used the idea of Caffarelli and Silvestre
\cite{cs}, which gives a formulation of the fractional Laplacian through Dirichlet-Neumann maps. Also in case of fractional $p$-Laplacian, existence and multiplicity results for polynomial type nonlinearities is studied by many authors see \cite{ssa,ssi,ls,lms,ms} and reference therein. Also eigenvalue problem related to $p-$fractional Laplacian is studied in \cite{gf, ll}.

\noi For $s=1$, the paper by Crandall, Robinowitz and Tartar \cite{cr} is the starting point on semilinear problem with singular nonlinearity. There is a large literature on singular nonlinearity see \cite{AJ, Am, cr, cp,dh, dhr, gr, gr1, JT, hss, hnc, lm,lmp, J1, j2, jh} and reference therein. In \cite{yc}, Chen showed the existence and multiplicity of the following problem

\begin{equation*}
  \left\{
\begin{array}{lr}
-\De w -\frac{\la}{|x|^2}w=  \frac{f(x)}{w^q} + \mu g(x) w^p \; \text{in}\;
\Om\setminus\{0\} \\
 \quad w>0\;\text{in}\;\Om\setminus\{0\}, \quad w = 0 \; \mbox{in}\; \partial \Om.\\
\end{array}
\quad \right.
\end{equation*}
where $0\in \Om$ is a bounded smooth domain of $\mb R^n$ with smooth boundary, $0<\la<\frac{(n-2)^2}{4}$, $0<q<1<p<\frac{n+2}{n-2}$, $f(x)>0$ and $g$ is sign-changing continuous function.

%
\noi To the best of our knowledge, there is no work related to system of fractional Laplacian with singular and sign-changing nonlinearity. In this work, we studied the multiplicity results for the system of fractional Laplacian equation with singular nonlinearity and sign-changing weight function with respect to the parameter $\la,$ $\mu$. This work is motivated by the work of Chen and Chen in \cite{yc}. But one can not directly extend all the results for fractional $p-$Laplacian, due to the non-local behavior of the operator and the bounded support of the test function is not preserved. Also due to the singularity of the problem, the associated functional is not differentiable in the sense of G\^{a}teaux.  The results obtained here are somehow expected but we show how the results arise out of nature of the Nehari manifold.

\noi The paper is organized as follows: Section 2 is devoted to some preliminaries and notations. we also state our main results. In section 3, we study the decomposition of Nehari manifold and the associated energy functional is bounded below and coercive. Section 3 contains the existence of a nontrivial solutions in $\mc N_{\la,\mu}^{+}$ and $\mc N_{\la,\mu}^{-}$.

\noi We will use the following notation throughout this paper: $\|f\|_{q^*}$, $\|g\|_{q^*}$ denote the norm in $L^{\frac{\al+\ba}{\al+\ba-1+q}}(\Om)$.
\section{Preliminaries:}
In this section we give some definitions and functional settings. At the end of this section, we state our main results. For this we define $H^{s}(\Om)$, the usual fractional Sobolev
space $H^{s}(\Om):= \left\{w\in L^{2}(\Om); \frac{(w(x)-w(y))}{|x-y|^{\frac{n}{2}+s}}\in L^{2}(\Om\times\Om)\right\}$ endowed with the norm
\begin{align}\label{2}
\|w\|_{H^{s}(\Om)}=\|w\|_{L^2(\Om)}+ \left(\int_{\Om\times\Om}
\frac{|w(x)-w(y)|^{2}}{|x-y|^{n+2s}}dxdy \right)^{\frac 12}.
\end{align}
To study fractional Sobolev space in details we refer \cite{hi}.

\noi Due to the non-localness of the operator, we define linear space as follows:
{\small\[X_0= \left\{w|\;w:\mb R^n \ra\mb R \;\text{is measurable},
w|_{\Om} \in L^p(\Om)\;
 and\;  \frac{w(x)- w(y)}{|x-y|^{\frac{n+2s}{2}}}\in
L^2(Q);  w = 0 \;\text{a.e. in}\; \mb R^n\setminus \Om\right\}\]}

\noi where $Q=\mb R^{2n}\setminus(\mc C\Om\times \mc C\Om)$ and
 $\mc C\Om := \mb R^n\setminus\Om$. The space $X_0$ was firstly introduced by Servadei and Valdinoci \cite{mp}. The space $X_0$ endowed with the norm
\begin{align}\label{01}
 \|w\|=\left(\int_{Q}\frac{|w(x)-w(y)|^{2}}{|x-y|^{n+2s}}dx dy\right)^{\frac12}
\end{align}
is a Hilbert space.
We notice that, the norms in \eqref{2} and \eqref{01} are not same because $\Om\times
\Om$ is strictly contained in $Q$. 
Let $Y= X_0\times X_0$ be the cartesian product of two reflexive Banach spaces, which is also reflexive Banach space with the norm
\[\|(u,w)\|=(\|u\|_{X_0}^{2}+\|w\|_{X_0}^{2})^{\frac 12}= \left(\int_{Q}\frac{|u(x)-u(y)|^{2}}{|x-y|^{n+2s}}dx dy +\int_{Q}\frac{|w(x)- w(y)|^{2}}{|x-y|^{n+2s}}dx dy\right)^{\frac12}.\]
Now we define the space
\[C_{Y}:= \{\ee: u,w \in C_{c}^{\infty}(\mb R^n): u=w=0\; \text{in}\; \mb R^n\setminus \Om\}.\]
Then $C_{\h}$ is a dense in the space $\h$.\\
\noi Denote $S:=\inf_{u\in X_{0}}\left\{\frac{\int_{\mb R^{2n}}|u(x)-u(y)|^{2}|x-y|^{-(n+2s)}dx dy}{(\int_{\mb R}|u|^{\al+\ba}dx)^{\frac{2}{\al+\ba}}}\right\}$,
$\overline{S}:=\inf_{w\in Y}\left\{\frac{\|\ee\|^2}{(\int_{\mb R}|u|^{\al}|w|^{\ba} dx)^{\frac{2}{\al+\ba}}}\right\}$
and
\[K_{\la,\mu} = \la\int_{\Om}f(x)(u_+)^{1-q}dx + \mu\int_{\Om}g(x)(w_+)^{1-q} dx.\]
\begin{Definition}
\noi A weak solution of the problem $(P_{\la,\mu})$ is a function $(u,w)\in \h$, $u,$ $w>0$ in $\Om$ such that for  every $(\phi,\psi)\in\h$
{\small\begin{align*}
\int_{Q}&\frac{(u(x)-u(y))(\phi(x)-\phi(y))}{|x-y|^{(n+2s)}}  dxdy+\int_{Q}\frac{(w(x)-w(y))(\psi(x)-\psi(y))}{|x-y|^{(n+2s)}}  dxdy
= \la\io f(x) (u^{-q} \phi)(x)dx \\
&+\mu\io g(x) (w^{-q} \psi)(x)dx   + \frac{\al}{\al+\ba}\io b(x) (u^{\al-1}v^{\ba} \phi)(x) dx+ \frac{\ba}{\al+\ba}\io b(x) (u^{\al}w^{\ba-1} \psi)(x) dx.
\end{align*}}
\end{Definition}
In order to present the existence of positive solution of $(P_{\la,\mu})$, we will consider the following problem
  \begin{equation*}
 (P_{\la,\mu}^{+}) \left\{
\begin{array}{lr}
(-\De)^s u =  \la f(x) u_{+}^{-q} +  \frac{\al}{\al+\ba}b(x) u_{+}^{\al-1}w_{+}^\ba\; \text{in}\;
\Om \\
(-\De)^s w =  \mu g(x) w_{+}^{-q}+ \frac{\ba}{\al+\ba} b(x) u_{+}^{\al} w_{+}^{\ba-1}\; \text{in}\;
\Om \\
 \quad \quad u, w>0\;\text{in}\;\Om, \quad u = w = 0 \; \mbox{in}\; \mb R^n \setminus\Om,\\
\end{array}
\quad \right.
\end{equation*}
where  $w_+:= \max\{w,0\}$, denote the positive part of $w$.
Then the function $\ee\in \h$, $u, w>0$ in $\Om\times\Om$ is a weak solution of the problem $(P_{\la,\mu}^{+})$ if for  every $(\phi,\psi)\in\h$
{\small\begin{align*}
\int_{Q}&\frac{(u(x)-u(y))(\phi(x)-\phi(y))}{|x-y|^{(n+2s)}}  dxdy+\int_{Q}\frac{(w(x)-w(y))(\psi(x)-\psi(y))}{|x-y|^{(n+2s)}}  dxdy = \la\io f(x) (u_{+}^{-q} \phi)(x)dx\\
 &+\mu\io g(x) (w_{+}^{-q} \psi)(x)dx   + \frac{\al}{\al+\ba}\io b(x) (u_{+}^{\al-1}v_{+}^{\ba} \phi)(x) dx+ \frac{\ba}{\al+\ba}\io b(x) (w_{+}^{\al}v_{+}^{\ba-1} \psi)(x) dx.
\end{align*}}
We note that if $\ee>0$ is a solution of $(P_{\la,\mu}^{+})$ then one can easily see that $\ee$ is also a solution $(P_{\la,\mu})$. To find the solution of $(P_{\la,\mu}^{+})$, we will use variational approach. So we define the associated functional $J_{\la,\mu}: \h \ra \mb [-\infty,\infty)$ as
\[J_{\la,\mu}(u,w) =  \frac{1}{2}\uw- \frac{1}{1-q}\io (\la f(x)u_{+}^{1-q}+ \mu g(x) w_{+}^{1-q}) dx -\frac{1}{\al+\ba} \io b(x) u_{+}^{\al}w_{+}^{\ba}dx. \]

\noi Here $J_{\la,\mu}$ is not bounded below on $Y$ but is bounded below on appropriate subset $\n$ of $Y$. Therefore in order to obtain the existence results, we introduce the Nehari manifold
\begin{equation*}
\n= \left\{(u,w)\in Y: \ld J_{\la,\mu}^{\prime}(u,w),(u,w)\rd=0
\right\}=\left\{(u,w)\in Y : \phi_{u,w}^{\prime}(1)=0\right\}
\end{equation*}
where $\ld\;,\; \rd$ denotes the duality between $Y$ and its dual
space. Thus $(u,w)\in \n$ if and only if
\begin{equation}\label{eq2}
\uw-\left(\fg\right)- \io\bw =0
\end{equation}
We note that $\n$ contains every solution of
($P_{\la,\mu}$). Now as we know that the Nehari manifold is closely
related to the behavior of the functions $\phi_{u,v}: \mb R^+\ra \mb R$
defined as $\phi_{u,w}(t)=J_{\la,\mu}(tu,tw)$. Such maps are called fiber
maps and were introduced by Drabek and Pohozaev in \cite{DP}. For
$(u,w)\in Y$, we have
\begin{align*}
\phi_{u,w}(t) &= \frac{t^2}{2} \uw- \frac{t^{1-q}}{1-q}\aw- \frac{2t^{\al+\ba}}{\al+\ba}\io\bw ,\\
\phi_{u,w}^{\prime}(t) &=  t\uw- { t^{-q}}\aw  - t^{\al+\ba-1} \io\bw,\\
\phi_{u,w}^{\prime\prime}(t) &= \uw +  qt^{-q-1} \aw - (\al+\ba-1) t^{\al+\ba-2}\io\bw.
\end{align*}
Then it is easy to see that $(tu,tw)\in \n$ if and only if
$\phi_{u,w}^{\prime}(t)=0$ and in particular, $u\in \n$ if
and only if $\phi_{u,w}^{\prime}(1)=0$. Thus it is natural to split
$\n$ into three parts corresponding to local minima,
local maxima and points of inflection. For this we set
\begin{align*}
\n^{\pm}&:= \left\{(u,w)\in \n:
\phi_{u,w}^{\prime\prime}(1)
\gtrless0\right\} =\left\{(tu,tw)\in Y : \phi_{u,w}^{\prime}(t)=0,\; \phi_{u,w}^{\prime\prime}(t)\gtrless  0\right\},\\
\n^{0}&:= \left\{(u,w)\in \n:
\phi_{u,w}^{\prime\prime}(1) = 0\right\}=\left\{(tu,tw)\in Y :
\phi_{u,w}^{\prime}(t)=0,\; \phi_{u,w}^{\prime\prime}(t)= 0\right\}.
\end{align*}
We also observe that if $\ee\in\n$ then
$$\phi_{u,w}^{\prime\prime}(1)
=\left\{
\begin{array}{lr}
(1+q) \uw - (\al+\ba-1+q)\io\bw\\
(2-\al-\ba) \uw +(\al+\ba-1+q) \aw.
\end{array}
\right.
$$
\noi Inspired by \cite{yc}, we show that how variational methods can be used to established some existence and multiplicity results for $(P_{\la,\mu}^{+})$. Our results are as follows:
\begin{Theorem}\label{th1}
Suppose that $\la\in (0,\La)$, where
\begin{align*}
\La:= \frac{(1+q)}{(\al+\ba-1+q)} \left(\frac{\al+\ba-2}{\al+\ba-1+q}\right)^{\frac{\al+\ba-2}{1+q}} \frac{1}{\|b\|} \left(\frac{S^{\al+\ba-1+q}}{\|a\|^{\al+\ba-2}}\right)^{\frac{1}{1+q}}.
\end{align*}
Then the problem $(P_{\la,\mu})$ has at least two solutions
\[\ee\in\mc N_{\la,\mu}^{+}, (U,W)\in \mc N_{\la,\mu}^{-}\;\mbox{ with}\; \|(U,W)\|>\|(u,w)\|.\]
\end{Theorem}
\section{Fibering map analysis}
In this section, we show that $\mc N_{\la,\mu}^{\pm}$ is nonempty and $\mc N_{\la,\mu}^{0}=\{(0,0)\}$. Moreover, $J_{\la,\mu}$ is bounded below and coercive. Define
{\small\begin{align}
\Ga :=\left\{(\la,\mu)\in \mb R^2\setminus\{(0,0)\}:0<\La:= (|\la|\|f\|_{q^*})^{\frac{2}{1+q}}+(|\mu|\|g\|_{q^*})^{\frac{2}{1+q}}< C(n,\al,\ba,q,S)\right\},
\end{align}}
where,
{\small\begin{align}
C(n,\al,\ba,q,S)&= \left(\frac{(1+q)}{(\al+\ba-1+q)}\right)^{\frac{2}{\al+\ba-2}} \left(\frac{\al+\ba-2}{\al+\ba-1+q}\right)^{\frac{2}{1+q}} \left(\frac{1}{\|b\|_{\infty}}\right)^{\frac{2}{\al+\ba-2}} S^{\frac{2(\al+\ba-1+q)}{(1+q)(\al+\ba-2)}}.
\end{align}}

\begin{Lemma}\label{le1}
Let $(\la,\mu)\in\Ga$. Then for each $\ee\in \h$ with $\aw>0$, we have the following:
\begin{enumerate}
\item[$(i)$] $ \io \bw\leq 0$, then there exists a unique $0<t_1<t_{max}$ such that $(t_1 u,t_1 w)\in \mc N_{\la,\mu}^{+}$ and $J_{\la,\mu}(t_1u,t_1 w)=\ds\inf_{t> 0} J_{\la,\mu}(tu,tw)$,
\item[$(ii)$] $\io \bw > 0$, then there exists a unique $t_1$ and $t_2$ with $0<t_1<t_{max}<t_2$ such that $(t_1u,t_1 w)\in \mc N_{\la,\mu}^{+} $, $(t_2 u,t_2 w)\in \mc N_{\la,\mu}^{-}$ and
$J_{\la,\mu}(t_1 u, t_1 w)=\ds\inf_{0\leq t\leq t_{max}} J_{\la,\mu}(tu, tw)$, $J_{\la,\mu}(t_2u, t_2 w)=\ds\sup_{t\geq t_1} J_{\la,\mu}(tu, tw)$.
\end{enumerate}
\end{Lemma}

\begin{proof}
For $t>0$, we define
\begin{align*}
\psi_{u,w}(t)= t^{2-\al-\ba}\uw -t^{-\al-\ba+1-q} \aw- \io\bw.
\end{align*}
One can easily see that $\psi_{u,w}(t)\ra -\infty$ as $t\ra 0^{+}$. Now
{\small\begin{align*}
\psi_{u,w}^{\prime}(t)&= (2-\al-\ba)t^{1-\al-\ba}\uw +(\al+\ba-1+q)t^{-\al-\ba-q} \aw.\\
\psi_{u,w}^{\prime\prime}(t)&= (2-\al-\ba)(1-\al-\ba)t^{-\al-\ba}\|\ee\|^2 -(\al+\ba-1+q)(\al+\ba+q)t^{-\al-\ba-q-1} \aw.
\end{align*}}
Then $\psi_{u,w}^{\prime}(t)=0$ if and only if $t=t_{max}:= \left[\frac{(\al+\ba-2)\uw}{(\al+\ba-1+q)\aw}\right]^{-\frac{1}{1+q}}$. Also
{\small\begin{align*}
\psi_{u,w}^{\prime\prime}&(t_{max})= (2-\al-\ba)(1-\al-\ba)\left[\frac{(\al+\ba-2)\uw}{(\al+\ba-1+q)\aw}\right]^{\frac{\al+\ba}{1+q}}\uw\\
&\quad - (\al+\ba-1+q)(\al+\ba+q)\left[\frac{(\al+\ba-2)\uw}{(\al+\ba-1+q)\aw}\right]^{\frac{\al+\ba+q+1}{1+q}} \aw \\
&= -\uw (\al+\ba-2)(1+q)\left[\frac{(\al+\ba-2)\uw}{(\al+\ba-1+q)\aw}\right]^{\frac{\al+\ba}{1+q}}<0.
\end{align*}}
Thus $\psi_{u,w}$ achieves its maximum at $t=t_{max}$. Now using the H\"{o}lder's inequality and fractional Sobolev inequality, we obtain
{\begin{align}\label{e2}
\aw \leq& |\la| \io |f(x)| |u|^{1-q}dx +|\mu| \io |g(x)| |w|^{1-q} dx\notag\\
\leq& |\la|\|f\|_{q^*} \|u\|_{\al+\ba}^{1-q}+ |\mu|\|g\|_{q^*} \|w\|_{\al+\ba}^{1-q}\notag\\
\leq& ((|\la|\|f\|_{q^*})^{\frac{2}{1+q}}+(|\mu|\|g\|_{q^*})^{\frac{2}{1+q}})^{\frac{1+q}{2}}\left(\frac{ \|\ee\|}{\sqrt{S}}\right)^{1-q}\\
= &\La^{\frac{1+q}{2}}\left(\frac{ \|\ee\|}{\sqrt{S}}\right)^{1-q}.
\end{align}}
\begin{align}\label{e3}
\io \bw \leq&   \|b\|_{\infty}\left(\frac{\al}{\al+\ba}\int_{\Om} |u|^{\al+\ba}dx +\frac{\ba}{\al+\ba}\int_{\Om} |v|^{\al+\ba} dx\right)\notag\\
\leq &\|b\|_{\infty} \left(\frac{\|\ee\|}{\sqrt{S}}\right)^{\al+\ba}.
\end{align}
Using \eqref{e2} and \eqref{e3} we obtain,
{\small\begin{align}\label{e40}
&\psi_{u,w}(t_{max})\notag\\
&= \frac{(1+q)}{(\al+\ba-1+q)} \left(\frac{\al+\ba-2}{\al+\ba-1+q}\right)^{\frac{\al+\ba-2}{1+q}} \frac{\|\ee\|^{\frac{2(\al+\ba-1+q)}{(1+q)}}}{[\aw]^{\frac{\al+\ba-2}{1+q}}}- \io\bw \notag\\
&\geq \left[\frac{(1+q)}{(\al+\ba-1+q)} \left(\frac{\al+\ba-2}{\al+\ba-1+q}\right)^{\frac{\al+\ba-2}{1+q}} \left(\frac{(\sqrt{S})^{(1-q)}}{\La^{\frac{1+q}{2}}}\right)
^{\frac{(\al+\ba-2)}{(1+q)}}  -  \|b\|_{\infty} \left(\frac{1}{\sqrt{S}}\right)^{\al+\ba} \right]\|\ee\|^{\al+\ba}\notag\\
&\equiv E_{\la,\mu} \|\ee\|^{\al+\ba}.
\end{align}}
where
{\small\begin{align*}
E_{\la,\mu} &= \left[\frac{(1+q)}{(\al+\ba-1+q)} \left(\frac{\al+\ba-2}{\al+\ba-1+q}\right)^{\frac{\al+\ba-2}{1+q}}\left(\frac{(\sqrt{S})^{(1-q)}}
{\La^{\frac{1+q}{2}}}\right)^{\frac{\al+\ba-2}{1+q}}  - \|b\|_{\infty} \left(\frac{1}{\sqrt{S}}\right)^{\al+\ba} \right]
\end{align*}}
Then we see that $E_{\la,\mu}=0$ if and only if $\La= C(n,\al,\ba,q,S)$, where
\begin{align*}
 C(n,\al,\ba,q,S)=  \left(\frac{(1+q)}{(\al+\ba-1+q)}\right)^{\frac{2}{\al+\ba-2}} \left(\frac{\al+\ba-2}{\al+\ba-1+q}\right)^{\frac{2}{1+q}} \left(\frac{1}{\|b\|_{\infty}}\right)^{\frac{2}{\al+\ba-2}} S^{\frac{2(\al+\ba-1+q)}{(1+q)(\al+\ba-2)}}.
\end{align*}
\noi Thus for $(\la,\mu)\in\Ga$, we have $E_{\la,\mu}>0$, and therefore it follows from \eqref{e40} that $\psi_{u,w}(t_{max})>0$. \\
\noi $(i)$ If $\io\bw\geq 0$, then $\psi_{u,w}(t)\ra - \io \bw<0$ as $t\ra\infty$. Consequently, $\psi_{u,w}(t)$ has exactly two points $0<t_{1}<t_{max}<t_{2}$ such that
\[ \psi_{u,w}(t_{1})=0=\psi_{u,w}(t_{2})\;\mbox{and}\; \psi^{\prime}_{u,w}(t_{1})>0> \psi^{\prime}_{u,w}(t_{2}).\]
Now we show that if $\psi_{u,w}(t)=0$ and $\psi^{\prime}_{u,w}(t)>0$, then $(tu,tw)\in\mc N_{\la,\mu}^{+}$.
 \begin{align*}
 \psi_{u,w}(t)=0  &\Leftrightarrow \|(tu,tw)\|^2 =  K_{\la,\mu}(tu,tw) +  \io b(x)(t u)_{+}^{\al}(t v)_{+}^{\ba} dx \\
 &\Leftrightarrow (tu,tw)\in \mc N_{\la,\mu},
 \end{align*}
 and therefore
{\small \begin{align*}
 \psi^{\prime}_{u,w}(t)>0 &\Rightarrow (2-\al-\ba)t^{1-\al-\ba}\uw -(-\al-\ba+1-q) t^{-\al-\ba-q} \aw >0\\
 &\Rightarrow (2-\al-\ba)\|(tu,tw)\|^2 +(\al+\ba-1+q) \left[\|(tu,tw)\|^2 -  \io b(x)(t u)_{+}^{\al}(t w)_{+}^{\ba} dx\right]>0,\\
 &\Rightarrow (1+q)\|(tu,tw)\|^2 - (\al+\ba-1+q)  \io b(x)(t u)_{+}^{\al}(t w)_{+}^{\ba} dx >0 \\
&\Rightarrow (tu,tw)\in \mc N_{\la,\mu}^{+}.
 \end{align*}}
Similarly one can show that if $\psi_{u,w}(t)=0$ and $\psi^{\prime}_{u,w}(t)<0$, then $(tu,tw)\in\mc N_{\la,\mu}^{-}$.\\
Now $\phi_{u,w}^{\prime}(t)= t^{\al+\ba-1}\psi_{u,w}(t)$. Thus $\phi_{u,w}^{\prime}(t)<0$ in $(0,t_1)$, $\phi_{u,w}^{\prime}(t)>0$ in $(t_1,t_2)$ and $\phi_{u,w}^{\prime}(t)<0$ in $(t_2,\infty)$. Hence $J_{\la,\mu}(t_1 u,t_1 w)=\ds\inf_{0\leq t\leq t_{max}} J_{\la,\mu}(tu, tw)$, $J_{\la,\mu}(t_1w,t_2 w)=\ds\sup_{t\geq t_1} J_{\la,\mu}(tu,t w)$. Moreover $(t_1u,t_{1} w)\in \mc N_{\la,\mu}^{+}$ and $(t_2u,t_{2} w)\in \mc N_{\la,\mu}^{-}$.\\
\noi $(ii)$ If $\io\bw<0$ and $\psi_{u,w}(t)\ra - \io \bw>0$ as $t\ra\infty$.  Consequently, $\psi_{u,w}(t)$ has exactly one point $0<t_{1}<t_{max}$ such that
\[ \psi_{u,w}(t_{1})=0\; \mbox{and}\;\psi^{\prime}_{u,w}(t_{1})>0.\]
Using $\phi_{u,w}^{\prime}(t)= t^{\al+\ba-1}\psi_{u,w}(t)$, we have $\phi_{u,w}^{\prime}(t)<0$ in $(0,t_1)$, $\phi_{u,w}^{\prime}(t)>0$ in $(t_1,\infty)$. So, $J_{\la,\mu}(t_1u,t_1 w)=\ds\inf_{t\geq 0} J_{\la,\mu}(tu, tw)$. Hence, it follows that $(t_1u,t_{1} w)\in \mc N_{\la,\mu}^{+}$.
\end{proof}
\begin{Corollary}
Suppose that $(\la,\mu)\in\Ga$, then $\mc N_{\la,\mu}^{\pm}\ne\emptyset$.
\end{Corollary}
\begin{proof}
From $(a1)$ and $(b1)$, we can
choose $\ee \in Y\setminus\{(0,0)\}$ such that $\aw>0$ and $\io\bw>0$.
By $(ii)$ of Lemma \ref{le1}, there exists unique $t_1$ and $t_2$ such that $(t_1u,t_1 w)\in \mc N_{\la,\mu}^{+}$, $(t_2u, t_2 w)\in \mc N_{\la,\mu}^{-}$. In conclusion, $\mc N_{\la,\mu}^{\pm}\ne \emptyset.$\QED
\end{proof}
 \begin{Lemma}
 For $(\la,\mu)\in\Ga$, we have $\mc N_{\la,\mu}^{0}=\{(0,0)\}$.
 \end{Lemma}
\begin{proof}
 We prove this by contradiction. Assume that there exists $(0,0)\not\equiv \ee\in \mc N_{\la,\mu}^{0}$. Then it follows from $\ee\in \mc N_{\la,\mu}^{0}$ that
\[(1+q)\uw =  (\al+\ba-1+q) \io\bw \]
and consequently
\begin{align*}
0&= \uw -\aw -\io\bw\\
&=\frac{(\al+\ba-2)}{(\al+\ba-1+q)}\uw -\aw.
\end{align*}

\noi Therefore, as $(\la,\mu)\in \Ga$ and $\ee\not\equiv (0,0)$, we use similar arguments as those in \eqref{e40} to get
{\small\begin{align*}
0&< E_{\la,\mu} \|\ee\|^{\al+\ba}\\
&\leq \frac{(1+q)}{(\al+\ba-1+q)}\left(\frac{\al+\ba-2}{\al+\ba-1+q}\right)^{\frac{\al+\ba-2}{1+q}} \frac{\|\ee\|^{\frac{2(\al+\ba-1+q)}{1+q}}}{\left[\aw\right]^{\frac{\al+\ba-2}{1+q}}}- \io\bw \\
&= \frac{(1+q)}{(\al+\ba-1+q)} \left(\frac{\al+\ba-2}{\al+\ba-1+q}\right)^{\frac{\al+\ba-2}{1+q}} \frac{\|\ee\|^{\frac{2(\al+\ba-1+q)}{1+q}}}{\left(\frac{\al+\ba-2}{\al+\ba-1+q}\uw\right)^{\frac{\al+\ba-2}{1+q}}}-\frac{(1+q)}{(\al+\ba-1+q)}\uw\\
&=0,
\end{align*}}
a contradiction. Hence $\ee=(0,0)$. That is, $\mc N_{\la,\mu}^{0}=\{(0,0)\}$.\QED
\end{proof}

\noi We note that $\Ga$ is also related to a gap structure in $\mc N_{\la,\mu}$:

\begin{Lemma}\label{le2}
Suppose that $(\la,\mu) \in \Ga$, then there exist a gap structure in $\mc N_{\la,\mu}$:
\[\|(U,W)\|> A_{0}> A_{\la,\mu}> \|(u,w)\|\; \mbox{for all} \; (u,w)\in \mc N_{\la,\mu}^{+}, (U,W)\in \mc N_{\la,\mu}^{-},\]
where
{\small\[A_{0}=  \left[\frac{(1+q)}{(\al+\ba-1+q)\|b\|_{\infty}} (\sqrt{S})^{\al+\ba} \right]^{\frac{1}{\al+\ba-2}}\;\mbox{and}\; A_{\la,\mu}=\left[\frac{(\al+\ba-1+q)}{(\al+\ba-2)} \left(\frac{1}{\sqrt{S}}\right)^{1-q} \right]^{\frac{1}{1+q}}\La^{\frac12}.\]}
\end{Lemma}
\begin{proof}
If $w\in \mc N_{\la,\mu}^{+}\subset \mc N_{\la,\mu}$, then
\begin{align*}
0&< (1+q)\uw - (\al+\ba-1+q)\io \bw \\
&= (2-\al-\ba) \uw + (\al+\ba-1+q)\aw.
\end{align*}
Hence it follows from \eqref{e2}
\begin{align*}
(\al+\ba-2) \uw &< (\al+\ba-1+q)\aw\\
& \leq (\al+\ba-1+q) ((|\la|\|f\|_{q^*})^{\frac{2}{1+q}}+(|\mu|\|g\|_{q^*})^{\frac{2}{1+q}})^{\frac{1+q}{2}}\left(\frac{ \|\ee\|}{\sqrt{S}}\right)^{1-q}
\end{align*}
which yields
\[\|\ee\|< \left[\frac{(\al+\ba-1+q) }{(\al+\ba-2)} \left(\frac{1}{\sqrt{S}}\right)^{1-q} \right]^{\frac{1}{1+q}}((|\la|\|f\|_{q^*})^{\frac{2}{1+q}}+(|\mu|\|g\|_{q^*})^{\frac{2}{1+q}})^{\frac{1}{2}}\equiv A_{\la,\mu}.\]
\noi If $(U,W)\in \mc N_{\la,\mu}^{-}$, then it follows from \eqref{e3} that
{\small\begin{align*}
(1+q)\|(U,W)\|^2 < (\al+\ba-1+q)\io b(x)U_{+}^{\al}W_{+}^{\ba} dx \leq  (\al+\ba-1+q)\|b\|_{\infty}\left(\frac{\|(U,W)\|}{\sqrt{S}}\right)^{\al+\ba}
\end{align*}}
which yields
\[\|(U,W)\|> \left[\frac{(1+q)}{(\al+\ba-1+q)\|b\|_{\infty}} (\sqrt{S})^{\al+\ba} \right]^{\frac{1}{\al+\ba-2}}\equiv A_0.\]
Now we show that $A_{\la,\mu}=A_0$ if and only if $\La= C(n,\al,\ba,q,S)$.
{\small\begin{align*}
 &\La = C(n,\al,\ba,q,S)= \left(\frac{(1+q)}{\|b\|_{\infty}(\al+\ba-1+q)}\right)^{\frac{2}{\al+\ba-2}} \left(\frac{\al+\ba-2}{\al+\ba-1+q}\right)^{\frac{2}{1+q}}  S^{\frac{2(\al+\ba-1+q)}{(1+q)(\al+\ba-2)}}.\\
&\Leftrightarrow A_{\la,\mu} =\La^{\frac{1}{2}}\left[\frac{(\al+\ba-1+q) }{(\al+\ba-2)} \left(\frac{1}{\sqrt{S}}\right)^{1-q} \right]^{\frac{1}{1+q}}\\
&= \left(\frac{(1+q)}{\|b\|_{\infty}(\al+\ba-1+q)}\right)^{\frac{1}{\al+\ba-2}} \left(\frac{\al+\ba-2}{\al+\ba-1+q}\right)^{\frac{1}{1+q}} S^{\frac{\al+\ba-1+q}{(1+q)(\al+\ba-2)}} \left[\frac{(\al+\ba-1+q) }{(\al+\ba-2)} \left(\frac{1}{\sqrt{S}}\right)^{1-q} \right]^{\frac{1}{1+q}}\\
  &\equiv A_0.
\end{align*}}
Thus for all $(\la,\mu)\in \Ga$, we can conclude that
\[\|(U,W)\|> A_{0}> A_{\la,\mu}> \|\ee\|\; \mbox{for all} \; \ee\in \mc N_{\la,\mu}^{+}, (U,W)\in \mc N_{\la,\mu}^{-}.\]
This completes the proof of the Lemma.\QED
\end{proof}

\begin{Lemma}\label{le3}
Suppose that $(\la,\mu)\in\Ga$, then $\mc N_{\la,\mu}^{-}$ is a closed set in $\h$- topology.
\end{Lemma}

\begin{proof}
Let $\{(U_k,W_k)\}$ be a sequence in $\mc N_{\la,\mu}^{-}$ with $(U_k,W_k) \ra (U,W)$ in $\h$. Then we have
\begin{align*}
 \|(U_k,W_k)\|^2 &=\lim_{k\ra\infty}\|(U_k,W_k)\|^2\\
&= \lim_{k\ra\infty} \left[\io (\la f(x)(U_k)_{+}^{1-q}+\mu g(x)(W_k)_{+}^{1-q}) dx + \io b(x)(U_{k})_{+}^{\al}(W_{k})_{+}^{\ba} dx\right]\\
&= \io (\la f(x)U_{+}^{1-q}+\mu g(x)W_{+}^{1-q}) dx + \io b(x)U_{+}^{\al}W_{+}^{\ba} dx
\end{align*}
and
\begin{align*}
(1+q) \|(U,W)\| - & (\al+\ba-1+q) \io b(x)U_{+}^{\al}W_{+}^{\ba} dx \\
& =\lim_{k\ra\infty} \left[(1+q) \|(U_k,W_k)\|^2 - (\al+\ba-1+q) \io b(x)(U_{k})_{+}^{\al}(W_{k})_{+}^{\ba} dx\right]\leq 0,
\end{align*}
i.e. $(U,W)\in \ \mc N_{\la,\mu}^{-}\cap \mc N_{\la,\mu}^{0}$. Since $\{(U_k,W_k)\}\subset \mc N_{\la,\mu}^{-}$, from Lemma \ref{le2} we have
\[\|(U,W)\|= \lim_{k\ra\infty} \|(U_k,W_k)\|\geq A_{\la,\mu}>0,\]
that is, $(U,W)\not\equiv (0,0)$. It follows from Lemma \ref{le1}, that $(U,W)\not\in \mc N_{\la,\mu}^{0}$ for any $(\la,\mu)\in\Ga$. Thus $(U,W)\in \mc N_{\la,\mu}^{-}$.
That is, $\mc N_{\la,\mu}^{-}$ is a closed set in $Y$- topology for any $(\la,\mu)\in\Ga$.\QED

\end{proof}


\begin{Lemma}\label{le4}
 Let $\ee\in \mc N_{\la,\mu}^{\pm}$, then for any $\Phi=(\phi,\psi)\in C_{Y}$, there exists a number $\e>0$ and a continuous function $f:B_{\e}(0):=\{v=(v_1,v_2)\in\h : \|v\|<\e\}\ra \mb R^{+}$ such that
\[f(v_1,v_2)>0, f(0,0)=1\;\mbox{and}\; f(v_1,v_2)(u+v_1\phi, w+v_2\psi)\in \mc N_{\la,\mu}^{\pm}\;\mbox{for all}\; v\in B_{\e}(0).\]
\end{Lemma}

\begin{proof}
We give the proof only for the case $\ee\in \mc N_{\la,\mu}^{+}$, the case $\mc N_{\la,\mu}^{-}$ may be preceded exactly. For any $C_{\h}$, we define $F:\h\times \mb R^{+}\ra \mb R$ as follows:
{\small \begin{align*}
F(v,t)&= t^{1+q}\|(u+v_1 \phi, w+v_2\psi)\|^2 - t^{\al+\ba-1+q}\io b(x) (u+v_1\phi)_{+}^{\al}(w+v_2\psi)_{+}^{\ba} dx\\
&\quad\quad-  K_{\la,\mu}(u+v_1\phi,w+v_2\psi)
\end{align*}}
Since $w\in \mc N_{\la,\mu}^{+}(\subset \mc N_{\la,\mu})$, we have that
\[F((0,0),1)= \uw- \aw - \io \bw =0,\]
and
\[\frac{\partial F}{\partial t}((0,0),1)= (1+q)\uw - (\al+\ba-1+q) \io \bw>0  .\]
Applying the implicit function Theorem at the point $((0,0),1),$ we have that there exists $\bar{\e}>0$ such that for $\|v\|<\bar{\e}$, $v\in \h$, the equation $F((v_1,v_2),t)=0$ has a unique continuous solution $t=f(v_1,v_2)>0.$ It follows from $F((0,0),1)=0$ that $f(0,0)=1$ and from $F((v_1,v_2),f(v_1,v_2))=0$ for $\|v\|<\bar\e$, $v\in\h$ that
{\small \begin{align*}
 0& = f^{1+q}(v) \|w+v\phi\|^2 - K_{\la,\mu}(u+v_1\phi,w+v_2\psi) - f^{\al+\ba-1+q}(v) \io b(x) (u+v_1\phi)_{+}^{\al} (w+v_2\psi)_{+}^{\ba} dx\\
 & = \frac{\|f(v)(u+v_1\phi,w+v_2\psi)\|^2-K_{\la,\mu}(f(v)(u+v_1\phi), f(v)(w+v_2\psi)) }{f^{1-q}(v)}\\
 &\quad\quad\quad-\frac{ \io b(x)(f(v)(u+v_1\phi))_{+}^{\al} (f(v)(w+v_2\psi))_{+}^{\ba}dx}{f^{1-q}(v)}
\end{align*}}
that is,
\[f(v_1,v_2)(u+v_1\phi,w+v_2\psi)\in \mc N_{\la,\mu}\;\mbox{for all}\; v\in \h, \|v\|<\tilde{\e}.\]
Since $\frac{\partial F}{\partial t}((0,0),1)>0$ and
{\small \begin{align*}
 &\frac{\partial F}{\partial t}((v_1,v_2),f(v_1,v_2))\\
 & = (1+q)f^{q}(v) \|(u+v_1\phi,w+v_2\psi\|^2 - (\al+\ba-1+q) f^{\al+\ba-1+q-1}(v) \io b(x)(u+v_1\phi)_{+}^{\al} (w+v_2\psi)_{+}^{\ba}\\
 & = \frac{(1+q)\|(f(v)(u+v_1\phi),f(v)(w+v_1\psi))\|^2}{f^{2-q}(v)}\\
 &\quad\quad-\frac{(\al+\ba-1+q) \io b(x)(f(v)(u+v_1\phi))_{+}^{\al} (f(v)(w+v_2\psi))_{+}^{\ba} dx}{f^{2-q}(v)}
\end{align*}}
we can take $\e>0$ possibly smaller $(\e<\bar\e)$ such that for any $v=(v_1,v_2)\in\h$, $\|v\|<\e$,
{\small\[ (1+q)\|(f(v)(u+v_1\phi),f(v)(w+v_2\psi))\|^2 - (\al+\ba-1+q)\io b(x)(f(v)(u+v_1\phi))_{+}^{\al} (f(v)(w+v_2\psi))_{+}^{\ba} dx >0, \]}
that is,
\[f(v_1,v_2)(u+v_1\phi, w+v_2\psi) \in \mc N_{\la,\mu}^{+}\;\mbox{for all}\; v=(v_1,v_2)\in B_{\e}(0).\]
This completes the proof of Lemma. \QED
\end{proof}

\begin{Lemma}\label{le5}
$J_\la$ is bounded below and coercive on $\mc N_{\la,\mu}$.
\end{Lemma}
\begin{proof}
For $\ee\in \mc N_{\la,\mu}$, we obtain from \eqref{e2} that
\begin{align}\label{s1}
J_{\la,\mu}(u,w)
&= \left(\frac{1}{2} -\frac{1}{\al+\ba}\right)\uw - \left(\frac{1}{1-q} -\frac{1}{\al+\ba}\right)\aw\notag\\
&\geq \left(\frac{1}{2} -\frac{1}{\al+\ba}\right)\uw - \left(\frac{1}{1-q} -\frac{1}{\al+\ba}\right) \La^{\frac{1+q}{2}} \left(\frac{\|\ee\|}{\mathrm{\sqrt{S}}}\right)^{1-q}.
\end{align}
Now consider the function $\rho: \mb R^{+}\ra \mb R$ as $\rho(t)= c t^2 - d t^{1-q}$, where $c,$ $d$ are both positive constants.
One can easily show that $\rho$ is convex($\rho^{\prime\prime}(t)>0$ for all $t>0$) with $\rho(t)\ra 0$ as $t\ra 0$ and $\rho(t)\ra \infty$ as $t\ra\infty$. $\rho$ achieves its minimum at $t_{min}=[\frac{d(1-q)}{2c}]^{\frac{1}{1+q}}$ and
\begin{align*}
\rho(t_{min})&= c\left[\frac{d(1-q)}{2c}\right]^{\frac{2}{1+q}}- d\left[\frac{d(1-q)}{2c}\right]^{\frac{1-q}{1+q}}= -\frac{(1+q)}{2} d^{\frac{2}{1+q}}\left(\frac{1-q}{2c}\right)^{\frac{1-q}{1+q}}.
\end{align*}
Applying $\rho(t)$ with $c=\left(\frac{1}{2} -\frac{1}{\al+\ba}\right)$, $d=\left(\frac{1}{1-q} -\frac{1}{\al+\ba}\right)\La^{\frac{1+q}{2}} \left(\frac{1}{\sqrt{S}}\right)^{1-q}$ and $t=\|\ee\|$, $\ee\in \mc N_{\la,\mu}$, we obtain from \eqref{s1} that
\[\lim_{\|\ee\|\ra\infty} J_{\la,\mu}(u,w)\geq \lim_{t\ra\infty} \rho(t) = \infty,\]
since $0<q<1$. That is $J_{\la,\mu}$ is coercive on $\mc N_{\la,\mu}$. Moreover it follows from \eqref{s1} that
\begin{align}\label{a1}
J_{\la,\mu}(u,w)\geq \rho(t)\geq \rho(t_{min})(\mbox{a constant}),
\end{align}
i.e
{\small \[J_{\la,\mu}(u,w)\geq -\frac{(1+q)}{2} d^{\frac{2}{1+q}}\left(\frac{1-q}{2c}\right)^{\frac{1-q}{1+q}}= -\frac{(1+q)(\al+\ba-2)}{(1-q)(\al+\ba)} \left(\frac{\al+\ba-1+q}{2(\al+\ba-2)}\right)^{\frac{2}{1+q}}\La \left(\frac{1}{\sqrt{S}}\right)^{\frac{2(1-q)}{1+q}}.\]}
Thus $J_{\la,\mu}$ is bounded below on $\mc N_{\la,\mu}$.\QED
\end{proof}
\section{Existence of Solutions in $\mc N_{\la,\mu}^{\pm}$}

Now from Lemma \ref{le3}, $\mc N_{\la,\mu}^{+}\cup \mc N_{\la,\mu}^{0}$ and $\mc N_{\la,\mu}^{-}$ are two closed sets in $\h$ provided $(\la,\mu)\in\Ga$. Consequently, the Ekeland variational principle can be applied to the problem of finding the infimum of $J_{\la,\mu}$ on both $\mc N_{\la,\mu}^{+}\cup \mc N_{\la,\mu}^{0}$ and $\mc N_{\la,\mu}^{-}$.
First, consider $\{(u_k,w_k)\}\subset\mc N_{\la,\mu}^{+}\cup \mc N_{\la,\mu}^{0}$  with the following properties:
\begin{align}
J_{\la,\mu}(u_k,w_k)&< \inf_{\ee\in\mc N_{\la,\mu}^{+}\cup \mc N_{\la,\mu}^{0}} J_{\la,\mu}(u,w) +\frac{1}{k},\label{c1}\\
J_{\la,\mu}(u,w)&\geq J_{\la,\mu}(u_k,w_k) -\frac{1}{k} \|(u-u_k,w-w_k)\|\;\mbox{for all}\; \ee\in \mc N_{\la,\mu}^{+} \cup \mc N_{\la,\mu}^{0}\label{c2}.
\end{align}
\begin{Lemma}
Show that the sequence $\{(u_k,w_k)\}$ is bounded in $\mc N_{\la,\mu}$. Moreover, there exists $0\not\equiv \ee \in \h$ such that $(u_k,w_k)\rightharpoonup \ee$ weakly in $\h$.
\end{Lemma}
\begin{proof}
From equations \eqref{a1} and \eqref{c1}, we have
\[c t^2 - d t^{1-q} = \rho(t)\leq J_{\la,\mu}\ee<\inf_{\ee\in\mc N_{\la,\mu}^{+}\cup \mc N_{\la,\mu}^{0}} J_{\la,\mu}\ee +\frac{1}{k}\leq C_5,\]
for sufficiently large $k$ and a suitable positive constant.
Hence putting $t= \|(u_k,w_k)\|$ in the above equation, we obtain $\{(u_k,w_k)\}$ is bounded.

Let $\{(u_k,w_k)\}$ is bounded in $\h$. Then, there exists a subsequence of $\{(u_k,w_k)\}_k$, still denoted by $\{(u_k,w_k)\}_k$ and $\ee\in \h$ such that
$(u_k,w_k)\rightharpoonup \ee \;\mbox{ weakly in}\; \h$, $(u_k,w_k)(\cdot)\ra \ee(\cdot)$ strongly in $(L^{r}(\Om))^2$ for $1\leq  r<p^{*}_s$ and $u_k(\cdot)\ra u(\cdot)$, $w_k(\cdot)\ra w(\cdot)$ a.e. in $\Om$.

\noi For any $\ee\in \mc N_{\la,\mu}^{+}$, we have from $0<q<1$, $2<\al+\ba<2^*_s$ that
\begin{align*}
J_{\la,\mu}(u,w)
&= \left(\frac{1}{2} -\frac{1}{1-q}\right)\uw + \left(\frac{1}{1-q} -\frac{1}{\al+\ba}\right)\io\bw\\
&< \left(\frac{1}{2} -\frac{1}{1-q}\right)\uw + \left(\frac{1}{1-q} -\frac{1}{\al+\ba}\right)\frac{1+q}{\al+\ba-1+q}\uw\\
&= \left(\frac{1}{\al+\ba} -\frac{1}{2}\right)\frac{(1+q)}{(1-q)}\uw<0,
\end{align*}
which means that $\inf_{\mc N_{\la,\mu}^{+}} J_{\la,\mu}<0$. Now for $(\la,\mu)\in\Ga$, we know from Lemma \ref{le1}, that $\mc N_{\la,\mu}^{0}=\{(0,0)\}$.
Together, these imply that $(u_k,w_k) \in \mc N_{\la,\mu}^{+}$ for $k$ large and
\[\inf_{\ee\in\mc N_{\la,\mu}^{+}\cup\mc N_{\la,\mu}^{0}}J_{\la,\mu}(u,w)= \inf_{\ee\in\mc N_{\la,\mu}^{+}} J_{\la,\mu}(u,w)<0.\]
Therefore, by weak lower semi-continuity of norm,
\[J_{\la,\mu}(u,w) \leq \liminf_{k\ra\infty} J_{\la,\mu}(u_k,w_k) = \inf_{\mc N_{\la,\mu}^{+}\cup\mc N_{\la,\mu}^{0}}J_{\la,\mu}<0, \]
that is, $\ee\not\equiv 0$ and $\ee\in Y$.\QED
\end{proof}
\begin{Lemma}\label{le6}
Suppose $(u_k,w_k)\in \mc N_{\la,\mu}^{+}$ such that $(u_k,w_k)\rightharpoonup \ee$ weakly in $\h$. Then for $(\la,\mu)\in \Ga$,
\begin{align}\label{s2}
(1+q) \io (\la f(x)u_{+}^{1-q}+\mu g(x)w_{+}^{1-q}) dx - (\al+\ba-2) \io b(x)u_{+}^{\al}w_{+}^{\ba} dx >0.
\end{align}
Moreover, there exists a constant $C_2>0$ such that
\begin{align}\label{s5}
(1+q) \uk - (\al+\ba-1+q) \io \bk \geq C_2 > 0.
\end{align}
\end{Lemma}
\begin{proof}
For $\{(u_k,w_k)\}\subset \mc N_{\la,\mu}^{+}(\subset \mc N_{\la,\mu})$, we have
\begin{align*}
&(1+q) \aw  - (\al+\ba-2) \io \bw\\
= &\lim_{k\ra\infty} \left[(1+q)\ak - (\al+\ba-2) \io \bk \right]\\
=&\lim_{k\ra\infty}\left[(1+q)\uk - (\al+\ba-1+q) \io\bk\right]\geq 0.
\end{align*}
Now, we can argue by a contradiction and assume that
\begin{align}\label{s3}
(1+q) \aw - (\al+\ba-2) \io \bw = 0.
\end{align}
Using $(u_k,w_k)\in \mc N_{\la,\mu}$, the weak lower semi continuity of norm and \eqref{s3} we have that
\begin{align*}
0= &\lim_{k\ra\infty} \left[\uk- \ak - \io \bk \right]\\
\geq & \uw- \aw dx -  \io \bw dx\\
=&\left\{\begin{array}{lr}
\uw-  \frac{\al+\ba-1+q}{1+q}\io \bw \\
\uw-  \frac{\al+\ba-1+q}{\al+\ba-2}\aw .
\end{array}\right.
\end{align*}qq
Thus for any $(\la,\mu)\in \Ga$ and $\ee\not\equiv 0$, by similar arguments as those in \eqref{e40} we have that
{\small\begin{align*}
0&< E_{\la,\mu}  \|\ee\|^{\al+\ba}\\
&\leq \frac{(1+q)}{(\al+\ba-1+q)} \left(\frac{\al+\ba-2}{\al+\ba-1+q}\right)^{\frac{\al+\ba-2}{1+q}} \frac{\|\ee\|^{\frac{2(\al+\ba-1+q)}{1+q}}}{\left[\aw\right]^{\frac{\al+\ba-2}{1+q}}}-
 \io \bw \\
&= \frac{(1+q)}{(\al+\ba-1+q)} \left(\frac{\al+\ba-2}{\al+\ba-1+q}\right)^{\frac{\al+\ba-2}{1+q}}\frac{\|\ee\|^{ \frac{2(\al+\ba-1+q)}{1+q}}}{\left(\frac{\al+\ba-2}{\al+\ba-1+q}\uw\right)^{\frac{\al+\ba-2}{1+q}}}-\frac{(1+q)}{(\al+\ba-1+q)}\uw\\
&=0,
\end{align*}}
which is clearly impossible. Now by \eqref{s2}, we have that
\begin{align}\label{s4}
(1+q)  \ak - (\al+\ba-2) \io \bk \geq C_2
\end{align}
for sufficiently large $k$ and a suitable positive constant $C_2$. This, together with the fact that $(u_k,w_k)\in \mc N_{\la,\mu}$ we obtain equation \eqref{s5}.\QED
\end{proof}

\noi Fix $(\phi,\psi)\in C_{\h}$ with $\phi,\psi\geq 0$. Then we apply Lemma \ref{le4} with $(u_k,w_k)\in \mc N_{\la,\mu}^{+}$ ($k$ large enough such that $\frac{(1-q)C_1}{k}<C_2$), we obtain a sequence of functions $f_k: B{\e_k}(0)\subset Y\ra \mb R$ such that
  $f_{k}(0,0)=1$ and $f_{k}(s_1,s_2)(u_k+s_1\phi,w_k+ s_2\psi)\in \mc N_{\la,\mu}^{+}$ for all $s=(s_1,s_2)\in B_{\e_k}(0)$. It follows from $(u_k,w_k)\in \mc N_{\la,\mu}$ and $f_{k}(s_1,s_2)(u_k+s_1\phi,w_k+ s_2\psi)\in \mc N_{\la,\mu}$ that
\begin{align}\label{f1}
\|(u_k,w_k)\|^2 - \ak - \io\bk dx=0
\end{align}
and
{\small\begin{align}\label{f2}
f_{k}^{2}(s_1,s_2)&\|(u_k+s_1\phi,w_k+s_2\psi)\|^2 - f_{k}^{1-q}(s_1,s_2) K(u_k+s_1\phi,w_k+s_2\phi) \notag\\
&- f_{k}^{\al+\ba}(s_1,s_2)\io b(x)(u_k+s_1\phi)_{+}^{\al}(w_k+s_2\psi)_{+}^{\ba} dx =0.
\end{align}}
 Choose $0<\rho <\e_k$, and $(s_1,s_2)=(\rho v_1,\rho v_2)$  with $\|v\|<1$ then we find $f_{k}(v_1,v_2)$ such that $f_{k}(0,0)=1$ and $f_{k}(v_1,v_2)(u_k+v_1\phi,w_k+ v_2\psi)\in \mc N_{\la,\mu}^{+}$ for all $v\in B_{\rho}(0)$.
\begin{Lemma}\label{le7}
For $(\la,\mu)\in \Ga$ we have $|\langle f^{\prime}_{k}(0,0), (v_1,v_2) \rangle|$ is finite for every $0\leq v=(v_1,v_2)\in C_{\h}$ with $\|v\|\leq1$.
\end{Lemma}
\begin{proof}
From \eqref{f1} and \eqref{f2} we have that
{\small\begin{align*}
0=&[f_{k}^2(\rho v_1,\rho v_2) -1]\|(u_k+\rho v_1\phi,w_k+ \rho v_2\psi)\|^2 +\|(u_k+\rho v_1\phi,w_k+\rho v_2\psi)\|^2-\|(u_k,w_k)\|^2\\
&- [f_{k}^{1-q}(\rho v_1,\rho_2) -1]\io (\la f(x)(u_k+\rho v_1\phi)_{+}^{1-q} +\mu g(x)(w_k+\rho v_2\psi)_{+}^{1-q})dx \\
&- \la\io f(x)[((u_k+v_1\phi)_{+}^{1-q}- (u_k)_{+}^{1-q})]dx- \mu\io g(x)[((w_k+v_2\psi)_{+}^{1-q}- (w_k)_{+}^{1-q})]dx\\
&- [f_{k}^{\al+\ba}(\rho v_1,\rho v_2) -1]\io b(x)(u_k+\rho v_1\phi)_{+}^{\al}(w_k+\rho v_2\psi)_{+}^{\ba}dx \\
&- \io b(x)[((u_k+\rho v_1\phi)^{\al}_{+}(w_k+\rho v_2\psi)_{+}^{\ba}- (u_k)_{+}^{\al}(w_k)_{+}^{\ba})]dx,\\
\leq & [f_{k}^2(\rho v_1,\rho v_2) -1]\|(u_k+\rho v_1\phi,w_k+ \rho v_2\psi)\|^2 +\|(u_k+\rho v_1\phi,w_k+\rho v_2\psi)\|^2-\|(u_k,w_k)\|^2\\
&- [f_{k}^{1-q}(\rho v_1,\rho_2) -1]\io (\la f(x)(u_k+\rho v_1\phi)_{+}^{1-q} +\mu g(x)(w_k+\rho v_2\psi)_{+}^{1-q})dx \\
&- [f_{k}^{\al+\ba}(\rho v_1,\rho v_2) -1]\io b(x)(u_k+\rho v_1\phi)_{+}^{\al}(w_k+\rho v_2\psi)_{+}^{\ba}dx \\
&- \io b(x)[((u_k+\rho v_1\phi)_{+}^{\al}(w_k+\rho v_2\psi)_{+}^{\ba}- (u_k)_{+}^{\al}(w_k)_{+}^{\ba})]dx,
\end{align*}}
since
\begin{align}\label{r1}
 (u_k+\rho v_1\phi)_{+}^{1-q}(x)- (u_k)_{+}^{1-q}(x) = &\left\{\begin{array}{lr}
 (u_k+\rho v_1\phi)^{1-q}(x)- (u_k)^{1-q}(x)\;\mbox{if}\; u_k\geq 0\\
 0\;\mbox{if}\; u_k\leq 0, u_k+ \rho v_1\phi\leq 0 \\
 (u_k+\rho v_1\phi)^{1-q}(x)\;\mbox{if}\; u_k\leq 0, u_k+\rho v_1\phi\geq 0,
\end{array}\right.
\end{align}
we have,
\[\io f(x)[((u_k+v_1\phi)_{+}^{1-q}- (u_k)_{+}^{1-q})(x)]dx\geq 0.\]
Similarly, one can see that
\[\io g(x)[((w_k+v_2\psi)_{+}^{1-q}- (w_k)_{+}^{1-q})(x)]dx\geq 0.\]
\noi Now dividing by $\rho>0$ and passing to the limit $\rho\ra 0$, we derive that
{\small \begin{align}\label{s6}
0&\leq\langle f_{k}^{\prime}(0,0),(v_1,v_2) \rangle\left[ 2 \|(u_k,w_k)\|^2 - (1-q) \ak -(\al+\ba)\io b(x)(u_k)_{+}^{\al}(w_k)_{+}^{\ba} dx\right]\notag\\
 &+ 2 \int_{Q}\frac{(u_k(x)-u_k(y))((v_1\phi)(x)-(v_1\phi)(y))+(w_k(x)-w_k(y))((v_2\psi)(x)-(v_2\psi)(y))}{|x-y|^{n+2s}} dxdy\notag\\
 &-\al \io b(x)(u_k)_{+}^{\al-1} (w_{k})_{+}^{\ba} v_1\phi dx -\ba \io b(x) (u_k)_{+}^{\al}(w_{k})_{+}^{\ba-1} v_2\psi dx\notag\\
=& \langle f_{k}^{\prime}(0,0),(v_1,v_2) \rangle\left[ (1+q)\|(u_k,w_k)\|^2 -(\al+\ba-1+q)\io b(x)(u_k)_{+}^{\al}(w_k)_{+}^{\ba}dx\right] \notag\\
 &+ 2 \int_{Q}\frac{(u_k(x)-u_k(y))((v_1\phi)(x)-(v_1\phi)(y))+(w_k(x)-w_k(y))((v_2\psi)(x)-(v_2\psi)(y))}{|x-y|^{n+2s}} dxdy\notag\\
 &-\al \io b(x)(u_k)_{+}^{\al-1} (w_{k})_{+}^{\ba} v_1\phi dx -\ba \io b(x) (u_k)_{+}^{\al}(w_{k})_{+}^{\ba-1} v_2\psi dx.
\end{align}}
From \eqref{s5} and \eqref{s6} we know immediately that $\langle f_{k}^{\prime}(0,0),(v_1,v_2) \rangle\ne -\infty$. Now we show that $\langle f_{k}^{\prime}(0,0),(v_1,v_2) \rangle\ne +\infty$. Arguing by contradiction, we assume that $\langle f_{k}^{\prime}(0,0),(v_1,v_2) \rangle=+\infty$. Since
\begin{align}\label{s7}
|f_k(\rho v_1,\rho_2)-1|&\|(u_k, w_k)\| + \rho f_{k}(\rho v_1,\rho v_2) \|(v_1 \phi, v_2\psi)\|\notag\\
& \geq \|[f_k(\rho v_1,\rho v_2)-1]\|(u_k,w_k\| + f_{k}(\rho v_1,\rho v_2)\|(\rho v_1 \phi, \rho v_2\psi)\|\notag\\
&= \|f_k(\rho v_1,\rho v_2) (u_k+\rho v_1\phi, w_k +\rho v_2\psi)- (u_k,w_k)\|
\end{align}
and
\[f_{k}(\rho v_1,\rho v_2)> f_{k}(0,0) =1\]
for sufficiently large $k$. From the definition of derivative $\langle f_{k}^{\prime}(0,0),(v_1,v_2) \rangle$, applying equation \eqref{c2} with $\ee= f_{k}(\rho v_1,\rho v_2)(u_k+\rho v_1\phi, w_k+\rho v_2\psi)\in \mc N_{\la,\mu}^{+}$, we
clearly have that
{\small\begin{align*}
&[f_k(\rho v_1,\rho v_2)-1]\frac{\|(u_k,w_k)\|}{k} +  f_{k}(\rho v_1,\rho v_2)\frac{\|\rho v \phi\|}{k}\\
& \geq \frac{1}{k}\|f_k(\rho v_1,\rho v_2)(u_k +\rho v_1 \phi, w_k +\rho v_2 \psi)-(u_k,w_k)\|\\
&\geq J_{\la,\mu}(u_k,w_k) - J_{\la,\mu}(f_{k}(\rho v_1,\rho v_2)(u_k +\rho v_1 \phi, w_k+ \rho v_2\psi))\\
&= \left(\frac{1}{2}-\frac{1}{1-q}\right) \|(u_k,w_k)\|^2+ \left(\frac{1}{1-q}-\frac{1}{2}\right) f_{k}^2(\rho v_1,\rho v_2)\|(u_k+\rho v_1\phi, w_k+\rho v_2\psi)\|^2 \\
& \quad+ \left(\frac{1}{1-q}-\frac{1}{\al+\ba}\right)\left(\io b(x) (u_k)_{+}^{\al}(w_{k})_{+}^{\ba} -  f_{k}^{\al+\ba}(\rho v_1,\rho v_2)\io b(x)(u_k+\rho v_1\phi)_{+}^{\al}(w_k+\rho v\phi)_{+}^{\ba} dx\right)\\
&= \left(\frac{1+q}{1-q}\right)(\|(u_k+\rho v_1\phi, w_k+\rho v_2\psi)\|^2-\|(u_k,w_k)\|^2 + [f_{k}^2(\rho v_1,\rho v_2)-1]\|(u_k+\rho v_1\phi, w_k+\rho v_2\psi)\|^2)\\
&\quad - \left(\frac{1}{1-q}-\frac{1}{\al+\ba}\right) f_{k}^{\al+\ba}(\rho v_1,\rho v_2) \io b(x) [((u_k+\rho v_1\phi)_{+}^{\al}(w_k +\rho v_2\psi)_{+}^{\ba} - (u_k)_{+}^{\al}(w_{k})_{+}^{\ba})] dx\\
&\quad - \left(\frac{1}{1-q}-\frac{1}{\al+\ba}\right) [f_{k}^{\al+\ba}(\rho v_1,\rho v_2)-1] \io b(x)(u_k)_{+}^{\al}(w_{k})_{+}^{\ba} dx.
\end{align*}}
Dividing by $\rho>0$ and passing to the limit as $\rho\ra 0$, we can obtain that
{\small\begin{align*}
&\langle f_{k}^{\prime}(0,0),(v_1,v_2) \rangle\frac{\|(u_k,w_k)\|}{k} + \frac{\|(v_1\phi, v_2\psi)\|}{k}\\
&\geq \left(\frac{1+q}{1-q}\right)\langle f_{k}^{\prime}(0,0),(v_1,v_2) \rangle\|(u_k,w_k)\|^2- \left(\frac{\al+\ba-1+q}{1-q}\right) \langle f_{k}^{\prime}(0,0),(v_1,v_2) \rangle \io b(x) (u_k)_{+}^{\al} (w_{k})_{+}^{\ba}\\
&+ \left(\frac{1+q}{1-q}\right)\int_{Q}\frac{(u_k(x)-u_k(y))((v_1\phi)(x)-(v_1\phi)(y))+(w_k(x)-w_k(y))((v_2\psi)(x)-(v_2\psi)(y))}{|x-y|^{n+2s}} dxdy\\
& - \left(\frac{\al+\ba-1+q}{1-q}\right)\left(\frac{\al}{\al+\ba}\io b(x)(u_{k})_{+}^{\al-1} (w_{k})_{+}^{\ba} v_1\phi  dx + \frac{\ba}{\al+\ba}\io b(x)(u_{k})_{+}^{\al} (w_{k})_{+}^{\ba-1} v_2\psi dx \right) \\
=& \frac{\langle f_{k}^{\prime}(0,0),(v_1,v_2) \rangle}{1-q}\left[ (1+q)\|(u_k,w_k)\|^2 - (\al+\ba-1+q) \io \bk \right]\\
& +\left(\frac{1+q}{1-q}\right)\int_{Q}\frac{(u_k(x)-u_k(y))((v_1\phi)(x)-(v_1\phi)(y))+(w_k(x)-w_k(y))((v_2\psi)(x)-(v_2\psi)(y))}{|x-y|^{n+2s}} dxdy\\
&- \left(\frac{\al+\ba-1+q}{1-q}\right) \left[\frac{\al}{\al+\ba} \io b(x) (u_k)_{+}^{\al-1} (w_{k})_{+}^{\ba}v_1\phi dx+ \frac{\ba}{\al+\ba}\io b(x) (u_k)_{+}^{\al} (w_{k})_{+}^{\ba-1}v_2\psi dx\right]
\end{align*}}
that is,
{\small\begin{align}\label{s8}
&\frac{\|(v_1\phi, v_2\psi)\|}{k}\notag\\
&\geq  \frac{\langle f_{k}^{\prime}(0,0),(v_1,v_2) \rangle}{1-q}\left[ (1+q)\|(u_k,w_k)\|^2 - (\al+\ba-1+q) \io \bk - \frac{(1-q)\|(u_k,w_k)\|}{k} \right]\notag\\
& +\left(\frac{1+q}{1-q}\right)\int_{Q}\frac{(u_k(x)-u_k(y))((v_1\phi)(x)-(v_1\phi)(y))+(w_k(x)-w_k(y))((v_2\psi)(x)-(v_2\psi)(y))}{|x-y|^{n+2s}} dxdy\notag\\
&- \left(\frac{\al+\ba-1+q}{1-q}\right)\left[ \frac{\al}{\al+\ba}\io b(x) (u_k)_{+}^{\al-1} (w_{k})_{+}^{\ba}v_1\phi dx+ \frac{\ba}{\al+\ba}\io b(x) (u_k)_{+}^{\al} (w_{k})_{+}^{\ba-1}v_2\psi\right]
\end{align}}
which is impossible because $\langle f_{k}^{\prime}(0,0),(v_1,v_2) \rangle=+\infty$ and
\[(1+q)\|(u_k,w_k)\|^2 -(\al+\ba-1+q)\io \bk -\frac{(1-q)\|(u_k,w_k)\|}{k} \geq C_2 - \frac{(1-q)C_1}{k}>0.\]
 In conclusion, $|\langle f_{k}^{\prime}(0,0),(v_1,v_2) \rangle|<+\infty$. Furthermore \eqref{s5} with $\|(u_k,w_k)\|\leq C_1$ and two inequalities \eqref{s6} and \eqref{s8} also imply that
\[|\langle f_{k}^{\prime}(0,0),(v_1,v_2) \rangle|\leq C_3\]
for $k$ sufficiently large and a suitable constant $C_3$.\QED
\end{proof}
\begin{Lemma}\label{le8}
For each $0\leq (\phi,\psi)\in C_{\h}$ and for every $0\leq v=(v_1,v_2)\in\h$ with $\|v\|\leq 1$, we have $\la f(x)u_{+}^{-q} v_1\phi+\mu g(x) w_{+}^{-q} v_2\psi \in L^{1}(\Om)$ and
{\small \begin{align}\label{s9}
\int_{Q}&\frac{(u(x)-u(y))((v_1\phi)(x)-(v_1\phi)(y))}{|x-y|^{n+2s}} dxdy+ \int_{Q}\frac{(w(x)-w(y))((v_2\phi)(x)-(v_2\phi)(y))}{|x-y|^{n+2s}} dxdy\notag\\
&- \io (\la f(x)u_{+}^{-q} v_1\phi +\mu g(x) w_{+}^{-q} v_2\psi) dx - \io b(x) u_{+}^{\al-1}v_{+}^{\ba}v_1\phi dx -\io b(x)u_{+}^{\al} w_{+}^{\ba-1}v_2\psi dx\geq 0.
\end{align}}
\end{Lemma}
\begin{proof}
Applying \eqref{s7} and $\eqref{c2}$ again, we have that
{\small\begin{align*}
&[f_k(\rho v_1,\rho v_2)-1]\frac{\|(u_k,w_k)\|}{k} +  f_{k}(\rho v_1,\rho v_2)\frac{\| \rho v\phi\|}{k}\\
& \geq \frac{1}{k}\|f_k(\rho v_1,\rho v_2)(w_k + \rho v \phi)-w_k\|\\
&\geq J_{\la,\mu}(u_k,w_k) - J_{\la,\mu}(f_k(\rho v_1,\rho v_2)(w_k+ \rho v\phi))\\
&= \frac{1}{2}\|(u_k,w_k)\|^2 -\frac{1}{2} \|f_{k}(\rho v_1,\rho v_2)(w_k+\rho v\phi)\|^2 dx -\frac{1}{1-q}\io (\la f(x) (u_k)_{+}^{1-q}+g(x) (w_k)_{+}^{1-q})dx\\
&\quad +\frac{1}{1-q}\io (\la f(x)(f_{k}(\rho v_1,\rho v_2)(u_k+\rho v_1\phi))_{+}^{1-q}+\mu g(x)(f_{k}(\rho v_1,\rho v_2)(w_k+\rho v_2\psi))_{+}^{1-q} )dx \\
&\quad-  \frac{1}{\al+\ba}\io  b(x)(u_k)_{+}^{\al}(w_k)_{+}^{\ba}dx + \frac{1}{\al+\ba} \io b(x)(f_{k}(\rho v_1,\rho v_2)(w_k+\rho v_1\phi))_{+}^{\al}(w_k+\rho v_2\psi))_{+}^{\ba}\\
&=- \frac{f^{2}_{k}(\rho v_1,\rho v_2)-1}{2} \|(u_k,w_k)\|^2 - \frac{f_{k}^2(\rho v_1,\rho v_2)}{2}(\|(u_k+\rho v_1\phi, w_k+\rho v_2\psi)\|^2- \|(u_k,w_k)\|^2)\\
&\quad +\frac{f_{k}^{1-q}(\rho v_1,\rho v_2)-1}{1-q} \io (\la f(x) (u_k+\rho v_1\phi)_{+}^{1-q}+\mu g(x) (w_k+\rho v_2\psi)_{+}^{1-q})dx\\
 &\quad+\frac{1}{1-q} \io a(x) [((w_k+\rho v\phi)_{+}^{1-q} - (w_{k})_{+}^{1-q})(x)] dx\\
&\quad + \frac{f_{k}^{\al+\ba}(\rho v_1,\rho v_2)-1}{\al+\ba} \io b(x) (u_k+\rho v_1\phi))_{+}^{\al}(w_k+\rho v_2\psi)_{+}^{\ba} dx\\
 &\quad+\frac{1}{\al+\ba} \io b(x) [((u_k+\rho v_1\phi)_{+}^{\al}(w_k+\rho v_2\psi)_{+}^{\ba} - (u_k)_{+}^{\al}(w_{k})_{+}^{\ba})(x)]dx.
\end{align*}}
\noi Dividing by $\rho>0$ and passing to the limit $\rho\ra 0^+$, we obtain
{\small \begin{align*}
&|\langle f^{\prime}_k(0,0),(v_1,v_2) \rangle|\frac{\|(u_k,w_k)\|}{k} + \frac{\|(v_1\phi,v_2\psi)\|}{k}\\
\geq 
&\;- \langle f^{\prime}_{k}(0,0), (v_1,v_2)\rangle\left[\|(u_k,w_k)\|^2 -\ak -\io\bk dx\right]\\
& -\int_{Q}\frac{(u_k(x)-u_k(y))(\phi(x)-\phi(y))+(w_k(x)-w_k(y))(\psi(x)-\psi(y))}{|x-y|^{n+2s}}dxdy \\
&\;+\frac{\al}{\al+\ba}\io b(x) (u_k)_{+}^{\al-1}(w_k)_{+}^{\ba} v_1\phi dx + \frac{\ba}{\al+\ba}+ \io b(x) (u_k)_{+}^{\al}(w_k)_{+}^{\ba-1} v_2\psi dx\\
&+ \frac{1}{1-q}\liminf_{\rho\ra 0^{+}}\left[ \io\frac{ \la f(x)((u_k+\rho v_1\phi)_{+}^{1-q} - (w_{k})_{+}^{1-q})}{\rho} + \io\frac{ \mu g(x)((w_k+\rho v_2\psi)_{+}^{1-q} - (w_{k})_{+}^{1-q})}{\rho}\right]\\
=& -\int_{Q}\frac{(u_k(x)-u_k(y))(\phi(x)-\phi(y))+(w_k(x)-w_k(y))(\psi(x)-\psi(y))}{|x-y|^{n+2s}}dxdy\\
&\;+\frac{\al}{\al+\ba}\io b(x) (u_k)_{+}^{\al-1}(w_k)_{+}^{\ba} v_1\phi dx + \frac{\ba}{\al+\ba}+ \io b(x) (u_k)_{+}^{\al}(w_k)_{+}^{\ba-1} v_2\psi dx\\
&+ \frac{1}{1-q}\liminf_{\rho\ra 0^{+}} \left[\io\frac{ \la f(x)((u_k+\rho v_1\phi)_{+}^{1-q} - (w_{k})_{+}^{1-q})}{\rho} + \io\frac{ \mu g(x)((w_k+\rho v_2\psi)_{+}^{1-q} - (w_{k})_{+}^{1-q})}{\rho}\right].
\end{align*}}
Then by above inequality, one can see that
{\small\[\liminf_{\rho\ra 0^{+}}\left[ \io\frac{ \la f(x)((u_k+\rho v_1\phi)_{+}^{1-q} - (w_{k})_{+}^{1-q})}{\rho}dx + \io\frac{ \mu g(x)((w_k+\rho v_2\psi)_{+}^{1-q} - (w_{k})_{+}^{1-q})}{\rho}dx\right]\]}
is finite. Now, using \eqref{r1}, we have
\[f(x)((w_k+ \rho v_1\phi)_{+}^{1-q} - (w_{k})_{+}^{1-q})\geq0\]
and similarly we have,
\[g(x)((w_k+\rho v_2\psi)_{+}^{1-q} - (w_{k})_{+}^{1-q})\geq 0,\;\mbox{ for all}\;  x\in \Om,\mbox{ for all} \;t>0.\]
Then by the Fatou Lemma, we have that
{\small\begin{align*}
&\io (\la f(x)(u_{k})_{+}^{-q}v_1\phi+\mu g(x)(w_{k})_{+}^{-q}v_2\psi)  dx\\
&\leq \frac{1}{1-q} \liminf_{\rho\ra 0^{+}} \left[\la\io\frac{ f(x)((u_k+\rho v_1\phi)_{+}^{1-q} - (u_{k})_{+}^{1-q})}{\rho}+ \mu \io\frac{ g(x) ((w_k+\rho v_2\psi)_{+}^{1-q} - (w_{k})_{+}^{1-q})}{\rho}dx\right]\\
&\leq \frac{|\langle f^{\prime}_k(0,0),(v_1,v_2) \rangle|\|(u_k,w_k)\|+ \|(v_1\phi, v_2\psi)\|}{k}\\
&\quad+\int_{Q}\frac{(u_k(x)-u_k(y))(\phi(x)-\phi(y))+(w_k(x)-w_k(y))(\psi(x)-\psi(y))}{|x-y|^{n+2s}}dxdy\\
&\quad- \frac{\al}{\al+\ba} \io b(x)(u_k)_{+}^{\al-1} (w_{k})_{+}^{\ba} v_1\phi dx
-\frac{\ba}{\al+\ba} \io b(x)(u_k)_{+}^{\al} (w_{k})_{+}^{\ba-1} v_2\psi dx\\
&\leq \frac{C_1C_3 \|(v_1,v_2)\|+ \|(v_1\phi, v_2\psi)\|}{k}- \frac{\al}{\al+\ba} \io b(u_k)_{+}^{\al-1} (w_{k})_{+}^{\ba} v_1\phi
-\frac{\ba}{\al+\ba} \io b(x)(u_k)_{+}^{\al} (w_{k})_{+}^{\ba-1} v_2\psi\\
&\quad+\int_{Q}\frac{(u_k(x)-u_k(y))(\phi(x)-\phi(y))+(w_k(x)-w_k(y))(\psi(x)-\psi(y))}{|x-y|^{n+2s}}dxdy
\end{align*}}
\noi Again using the Fatou Lemma and the above relation we have
{\small\begin{align*}
 &\la \io f(x) u_{+}^{-q} v_1\phi dx - \mu \io g(x) w_{+}^{-q}v_2\psi dx \leq \io \left[\liminf_{k\ra\infty}(\la f(x) u_{+}^{-q}v_1\phi + \mu \io g(x) w_{+}^{-q}v_2\psi) dx \right] dx\\
 &\leq  \liminf_{k\ra\infty}\io (\la f(x)(u_{k})_{+}^{-q}v_1\phi+\mu(w_{k})_{+}^{-q}v_2\psi) dx \\
&\leq \frac{C_1C_3 \|(v_1,v_2)\|+ \|(v_1\phi, v_2\psi)\|}{k}- \frac{\al}{\al+\ba} \io b(x)(u_k)_{+}^{\al-1} (w_{k})_{+}^{\ba} v_1\phi
-\frac{\ba}{\al+\ba} \io b(u_k)_{+}^{\al} (w_{k})_{+}^{\ba-1} v_2\psi\\
&\quad+\int_{Q}\frac{(u_k(x)-u_k(y))(v_1\phi(x)-v_1\phi(y))+(w_k(x)-w_k(y))(v_2\psi(x)-v_2\psi(y))}{|x-y|^{n+2s}}dxdy
\end{align*}}
which completes the proof of Lemma.\QED
\end{proof}
\begin{Corollary}\label{cc1}
For every $0\leq (\phi,\psi)\in \h$, we have $(\la f(x) u_{+}^{-q}\phi+\mu g(x) w_{+}^{-q}\psi) \in L^{1}(\Om)$, $u_{+}$, $w_{+}>0$ in ${\Om}$ and
{\small \begin{align}\label{s91}
\int_{Q}&\frac{(u(x)-u(y))(\phi(x)-\phi(y))}{|x-y|^{n+2s}} dxdy+\int_{Q}\frac{(w(x)-w(y))(\psi(x)-\psi(y))}{|x-y|^{n+2s}} dxdy-\la \io f(x) u_{+}^{-q}\phi dx \notag\\
&- \mu \io g(x) w_{+}^{-q}\psi dx-\frac{\al}{\al+\ba} \io b(x) u_{+}^{\al-1}w_{+}^{\ba}\phi dx-\frac{\ba}{\al+\ba} \io b(x) u_{+}^{\al}w_{+}^{\ba-1}\psi dx\geq 0.
\end{align}}
\end{Corollary}
\begin{proof}
Choosing $v=(v_1,v_2)\in\h$ such that $v\geq 0$, $v\equiv l$ in the neighborhood of support of $\phi$ and $\|v\|\leq 1$, for some $l>0$ is a constant. Then we note that $ \la\io f(x) u_{+}^{-q}\phi dx + \mu \io g(x) w_{+}^{-q}\psi dx<\infty$, for every $0\leq (\phi,\psi)\in C_{Y}$ which guarantees that $u_{+}$, $w_{+}>0$  a.e in $\Om$. Putting this choice of $v$ in \eqref{s9}, we have for every $0\leq (\phi,\psi)\in C_{Y}$
{\small
\begin{align*}
\int_{Q}&\frac{(u(x)-u(y))(\phi(x)-\phi(y))}{|x-y|^{n+2s}} dxdy+\int_{Q}\frac{(w(x)-w(y))(\psi(x)-\psi(y))}{|x-y|^{n+2s}} dxdy-\la \io f(x) u_{+}^{-q}\phi dx \notag\\
&- \mu \io g(x) w_{+}^{-q}\psi dx-\frac{\al}{\al+\ba} \io b u_{+}^{\al-1}w_{+}^{\ba}\phi dx-\frac{\ba}{\al+\ba} \io b u_{+}^{\al}w_{+}^{\ba-1}\psi dx\geq 0.
\end{align*}
}
Hence by density argument, \eqref{s91} holds for every $0\leq (\phi,\psi)\in \h$, which completes the proof of Corollary.
\end{proof}

\begin{Lemma}\label{le9}
We show that $u>0$, $w>0$ and $\ee\in \mc N_{\la,\mu}^{+}$.
\end{Lemma}
\begin{proof}
Using \eqref{s91} with $\phi=u^-$, $\psi=w^-$, we obtain that
\begin{align*}
0&\leq \int_{Q}\frac{(u(x)-u(y))(u^-(x)- u^-(y))}{|x-y|^{n+2s}} dxdy+ \int_{Q}\frac{(w(x)-w(y))(w^-(x)- w^-(y))}{|x-y|^{n+2s}} dxdy \\
&\leq - \|u^-\|^2 - \|w^-\|^2 - 2\int_{Q}\frac{u^-(x) u^+(y)+w^-(x) w^+(y)}{|x-y|^{n+2s}}dxdy \leq - \|u^-\|^2 - \|w^-\|^2 \leq 0.
\end{align*}
i.e, $u^{-}=w^-=0$ a.e. So, $u=u^+ >0$, $w=w^+ >0$ a.e by Corollary \ref{cc1}. Hence $u$, $w>0$ in $\Om$.
Now using \eqref{s91} with $\phi=u$, $\psi=w$, we obtain that
\begin{align*}
\uw\geq  \io (\la f(x) u_{+}^{1-q}+\mu g(x) w_{+}^{1-q}) dx + \io \bw.
\end{align*}
On the other hand, by the weak lower semi-continuity of the norm, we have that
\begin{align*}
\uw&\leq \liminf_{k\ra\infty}\|(u_k,w_k)\|^2\leq \limsup_{k\ra\infty}\|(u_k,w_k)\|^2\\
&= \io (\la f(x) u_{+}^{1-q}+\mu g(x) w_{+}^{1-q}) dx + \io \bw.
\end{align*}
Thus
\begin{align}\label{s10}
\uw
= \io (\la f(x) u_{+}^{1-q}+\mu g(x) w_{+}^{1-q})dx +  \io \bw.
\end{align}
Consequently, $(u_k,w_k)\ra (u,w)$ in $\h$ and $\ee\in \mc N_{\la,\mu}$. Now from \eqref{s2} it follows that
\begin{align*}
(1+q)\uw& -(\al+\ba-1+q)\io \bw\\
=& (1+q)\io (\la f(x) u_{+}^{1-q}+\mu g(x) w_{+}^{1-q}) dx - (\al+\ba-2) \io \bw>0,
\end{align*}
that is, $\ee\in \mc N_{\la,\mu}^{+}$. \QED
\end{proof}
\begin{Lemma}\label{le10}
Show that $w$ is in fact a positive weak solution of problem $(P_{\la,\mu})$.
\end{Lemma}
\begin{proof}
Let $(u,w)=(u_1,u_2)$, $(\phi_1,\phi_2)\in\h$ and $\e>0$, then we define \[\Psi(x)=(\Psi_1,\Psi_2)= ((u_1+\e\phi_1)_{+},(u_2+\e\phi_2)_{+})\]
For $i=1,2$, let $\Om =\Om_i\times\Ga_i$  with
\[\Om_i:= \{x\in \Om: u_i(x)+ \e \phi_i(x)>0\}\;\mbox{and}\;\Ga_i:= \{x\in\Om: u_i(x)+ \e \phi_i(x)\leq0\}.\]
Then $\Psi_i|_{\Om_i}(x)= (u_i+\e\phi)_{+}(x)$, and $\Psi_i|_{\Ga_i}(x)= 0$. Decompose
$$Q := (\Om_{i}\times \Om^c)\cup (\Ga_i\times \Om^c)\cup(\Om^c\times \Om_i)\cup(\Om^c\times \Ga_i)\cup(\Ga_i\times \Om_i)\cup(\Om_i\times \Ga_i)\cup(\Om_i\times \Om_i)\cup(\Ga_i\times \Ga_i).$$
Let $ M_i(x,y)= u_i(x,y)((u_i+\e\phi)^{-}(x)- (u_i+\e\phi)^{-}(y))K(x,y)$, where $u_i(x,y)=(u_i(x)-u_i(y))$ and $K(x,y)=\frac{1}{|x-y|^{n+2s}}$. Then we have
{\small\begin{enumerate}
\item[1.] $\int_{\Om_i\times \Om^c} M_i(x,y)dx dy = \int_{\Om^c\times \Om_i} M_i(x,y)dx dy= 0.$
\item[2.] $\int_{\Ga_i\times \Om^c} M_i(x,y)dx dy = -\int_{\Ga_i\times \Om^c} u_i(x)(u_i+\e\phi_i)(x)K(x,y)dxdy.$
\item[3.] $\int_{\Om^c\times \Ga_i} M_i(x,y)dx dy = -\int_{\Om^c\times \Ga_i} u_i(x)(u_i+\e\phi_i)(x)K(x,y)dxdy.$
\item[4.] $\int_{\Ga_i\times \Om_i} M_i(x,y)dx dy = - \int_{\Ga_i\times \Om_i}u_i(x,y)(u_i+\e\phi_i)(x)K(x,y)dx dy .$
\item[5.] $\int_{\Om_i\times \Ga_i} M_i(x,y)dx dy = - \int_{\Om_i\times \Ga_i}u_i(x,y)(u_i+\e\phi_i)(x)K(x,y)dx dy.$
\item[6.] $\int_{\Om_i\times \Om_i} M_i(x,y)dx dy = 0.$
\item[7.] $\int_{\Ga_i\times \Ga_i} M_i(x,y)dx dy = - \int_{\Ga_i\times \Ga_i}u_i(x,y)((u_i+\e\phi_i)(x)- (u_i+\e\phi_i)(y))K(x,y)dx dy.$
\end{enumerate}}
\noi Now relabeling $(\psi_1,\psi_2)=(\Phi,\Psi)$, $(u_1,u_2)=(u,w)$ and $(\phi_1,\phi_2)=(\phi,\psi)$. Then putting $(\Phi,\Psi)$ into \eqref{s9} and using \eqref{s10}, we see that
{\small \begin{align*}
0\leq & \int_{Q}\frac{u(x,y)(\Phi(x)-\Phi(y))+w(x,y)(\Psi(x)-\Psi(y))}{|x-y|^{n+2s}}dxdy- \io (\la f(x) u_{+}^{-q}\Phi+ \mu g(x) w_{+}^{-q}\Psi) dx\\
 &- \frac{\al}{\al+\ba}\io b(x) u_{+}^{\al-1}w_{+}^{\ba}\Phi dx-\frac{\ba}{\al+\ba}\io b(x) u_{+}^{\al}w_{+}^{\ba-1}\Psi dx\\
=& \int_{Q}\frac{u(x,y)((u+\e\phi)(x)- (u+\e\phi)(y))+ w(x,y)((w+\e\psi)(x)- (w+\e\psi)(y))}{|x-y|^{n+2s}}dxdy\\
&\quad+\int_{Q}\frac{u(x,y)((u+\e\phi)^{-}(x)- (u+\e\phi)^{-}(y))+ w(x,y)((w+\e\psi)^{-}(x)- (w+\e\psi)^{-}(y))}{|x-y|^{n+2s}}dxdy\\
&- \int_{\Om}(\la f(x) u_{+}^{-q}(u+\e\phi)+\mu g(x) w_{+}^{-q}(w+\e\psi)) - \int_{\Om}(\la f(x) u_{+}^{-q}(u+\e\phi)^{-}+ \mu g(x) w_{+}^{-q}(w+\e\phi)^{-})\\
 &- \frac{\al}{\al+\ba}\io b(x) u_{+}^{\al-1}w_{+}^{\ba}(u+\e\phi) dx-\frac{\ba}{\al+\ba}\io b(x) u_{+}^{\al}w_{+}^{\ba-1}(w+\e\phi) dx\\
 &- \frac{\al}{\al+\ba}\io b(x) u_{+}^{\al-1}w_{+}^{\ba}(u+\e\phi)^{-} dx-\frac{\ba}{\al+\ba}\io b(x) u_{+}^{\al}w_{+}^{\ba-1}(w+\e\phi)^{-} dx
\end{align*}}
{\small \begin{align*}
=&\e\left(\int_{Q}\frac{u(x,y)(\phi(x)- \phi(y))+w(x,y)(\psi(x)- \psi(y))}{|x-y|^{n+2s}} dxdy - \io (\la f(x) u_{+}^{-q}\phi + \mu g(x)
w_{+}^{-q}\psi)dx\right.\\                                                                                                                             
& \left.- \frac{\al}{\al+\ba}\io b(x) u_{+}^{\al-1}w_{+}^{\ba}\phi -\frac{\ba}{\al+\ba}\io b(x)                                                        
u_{+}^{\al}w_{+}^{\ba-1}\psi\right)+\int_{Q}\frac{|u(x)-u(y)|^{2}+|w(x)-w(y)|^{2}}{|x-y|^{n+2s}} dxdy\\                                                
 &\; +\int_{Q}\frac{u(x,y)((u+\e\phi)^{-}(x)- (u+\e\phi)^{-}(y))+w(x,y)((w+\e\phi)^{-}(x)- (w+\e\phi)^{-}(y))}{|x-y|^{n+2s}}dxdy \\                    
& - \io (\la f(x) w_{+}^{1-q}+ \mu g(x) u_{+}^{1-q}) dx+\la \int_{\Ga_1} f(x) u_{+}^{-q}(u+\e\phi)dx+ \mu \int_{\Ga_2}g(x) w_{+}^{-q}(w+\e\phi) dx\\   
&-\frac{\al}{\al+\ba} \int_{\Ga_1} b(x) u_{+}^{\al-1}w_{+}^{\ba}(u+\e\phi) dx-\frac{\ba}{\al+\ba} \int_{\Ga_2} b(x) u_{+}^{\al}w_{+}^{\ba-1}(w+\e\psi) 
dx                                                                                                                                                     
-  \io b(x) u_{+}^{\al}w_{+}^{\ba} dx \\                                                                                                               
&=\e\left(\int_{Q}\frac{u(x,y)(\phi(x)- \phi(y))+w(x,y)(\psi(x)- \psi(y))}{|x-y|^{n+2s}} dxdy - \io (\la f(x) u_{+}^{-q}\phi + \mu g(x)                
w_{+}^{-q}\psi)dx\right.\\                                                                                                                             
& \left.- \frac{\al}{\al+\ba}\io b(x) u_{+}^{\al-1}w_{+}^{\ba}\phi -\frac{\ba}{\al+\ba}\io b(x) u_{+}^{\al}w_{+}^{\ba-1}\phi                           
\right)-2\int_{\Ga_1\times \Om^c}\frac{u(x)(u+\e\phi)(x)}{|x-y|^{n+2s}}\\                                                                              
&-2\int_{\Ga_2\times \Om^c}\frac{w(x)(w+\e\phi)(x)}{|x-y|^{n+2s}}dxdy- 2\int_{\Ga_1\times \Om_1}\frac{u(x,y)(u+\e\phi)(x)}{|x-y|^{n+2s}}dx dy\\        
&- 2\int_{\Ga_2\times \Om_2}\frac{w(x,y)(w+\e\phi)(x)}{|x-y|^{n+2s}}dx dy-2\int_{\Ga_1\times\Ga_1}\frac{u(x,y)((u+\e\phi)(x)-                          
(u+\e\phi)(y))}{|x-y|^{n+2s}}\\                                                                                                                        
&-2\int_{\Ga_2\times\Ga_2}\frac{w(x,y)((w+\e\phi)(x)- (w+\e\phi)(y))}{|x-y|^{n+2s}}dxdy                                                                
+ \la\int_{\Ga_1} f(x) u_{+}^{-q}(u+\e\phi)dx\\                                                                                                        
&+ \mu \int_{\Ga_2}g(x) w_{+}^{-q}(w+\e\psi)) - \frac{\al}{\al+\ba}\int_{\Ga_1} b(x) u_{+}^{\al-1}w_{+}^{\ba}(u+\e\phi) -                              
\frac{\ba}{\al+\ba}\int_{\Ga_2} b(x) u_{+}^{\al}w_{+}^{\ba-1}(w+\e\psi) \\                                                                               
=&\e\left(\int_{Q}\frac{u(x,y)(\phi(x)- \phi(y))+w(x,y)(\psi(x)- \psi(y))}{|x-y|^{n+2s}} dxdy - \io (\la f(x) u_{+}^{-q}\phi + \mu g(x) w_{+}^{-q}\psi)dx\right.\\
& \left.- \frac{\al}{\al+\ba}\io b(x) u_{+}^{\al-1}w_{+}^{\ba}\phi dx-\frac{\ba}{\al+\ba}\io b(x) u_{+}^{\al}w_{+}^{\ba-1}\phi dx\right) -2\int_{\Ga_1\times \Om^c}\frac{|u(x)|^{2}}{|x-y|^{n+2s}}dx dy\\
&- 2\int_{\Ga_2\times \Om^c}\frac{|w(x)|^{2}}{|x-y|^{n+2s}}dx dy-2\int_{\Ga_1\times \Om_1}\frac{u(x,y)u(x)}{|x-y|^{n+2s}}dx dy-2\int_{\Ga_2\times \Om_2}\frac{w(x,y)w(x)}{|x-y|^{n+2s}}dx dy\\
&-2\int_{\Ga_1\times \Ga_1}\frac{|u(x)-u(y)|^{2}}{|x-y|^{n+2s}}dx dy -2\int_{\Ga_2\times \Ga_2}\frac{|w(x)-w(y)|^{2}}{|x-y|^{n+2s}}dx dy+ \la\int_{\Ga_1} f(x) u_{+}^{-q}(u+\e\phi) dx\\
&+\mu \int_{\Ga_2} g(x) w_{+}^{-q}(w+\e\psi)) - \frac{\al}{\al+\ba}\int_{\Ga_1} b(x) u_{+}^{\al-1}w_{+}^{\ba}(u+\e\phi) - \frac{\ba}{\al+\ba}\int_{\Ga_2} b(x) u_{+}^{\al}w_{+}^{\ba-1}(w+\e\psi) \\
&-2\e\left(\int_{\Ga_1\times \Om^c}\frac{ u(x)\phi(x)}{|x-y|^{n+2s}}dxdy+\int_{\Ga_2\times \Om^c}\frac{w(x)\psi(x)}{|x-y|^{n+2s}}dxdy+ \int_{\Ga_1\times \Om_1}\frac{u(x,y)\phi(x)}{|x-y|^{n+2s}}dx dy\right.\\
&\;+  \int_{\Ga_2\times \Om_2}\frac{ w(x,y)\psi(x)}{|x-y|^{n+2s}}dx dy+\left.\int_{\Ga_1\times \Ga_1}\frac{u(x,y)(\phi(x)- \phi(y))}{|x-y|^{n+2s}}+ \int_{\Ga_2\times \Ga_2}\frac{w(x,y)(\psi(x)- \psi(y))}{|x-y|^{n+2s}}dxdy\right)\\
\leq &\e\left(\int_{Q}\frac{u(x,y)(\phi(x)- \phi(y))+w(x,y)(\psi(x)- \psi(y))}{|x-y|^{n+2s}} dxdy - \io (\la f(x) u_{+}^{-q}\phi + \mu g(x) w_{+}^{-q}\psi)dx\right.\\
& \left.- \frac{\al}{\al+\ba}\io b(x) u_{+}^{\al-1}w_{+}^{\ba}\phi dx-\frac{\ba}{\al+\ba}\io b(x) u_{+}^{\al}w_{+}^{\ba-1}\phi\right)-2\int_{\Ga_1\times \Om_1}\frac{u(x,y)u(x)}{|x-y|^{n+2s}}dx dy
\end{align*}}
{\small\begin{align*}
&-2\int_{\Ga_2\times \Om_2}\frac{w(x,y)w(x)}{|x-y|^{n+2s}}dx dy+ \la\int_{\Ga_1} f(x) u_{+}^{-q}(u+\e\phi)+\mu \int_{\Ga_2} g(x) w_{+}^{-q}(w+\e\psi))
dx\\                                                                                                                                                   
&- \frac{\al}{\al+\ba}\int_{\Ga_1} b(x) u_{+}^{\al-1}w_{+}^{\ba}(u+\e\phi) dx- \frac{\ba}{\al+\ba}\int_{\Ga_2} b(x) u_{+}^{\al}w_{+}^{\ba-1}(w+\e\psi) 
dx\\                                                                                                                                                   
&-2\e\left(\int_{\Ga_1\times \Om^c}\frac{ u(x)\phi(x)}{|x-y|^{n+2s}}dxdy+\int_{\Ga_2\times \Om^c}\frac{w(x)\psi(x)}{|x-y|^{n+2s}}dxdy+                 
\int_{\Ga_1\times \Om_1}\frac{u(x,y)\phi(x)}{|x-y|^{n+2s}}dx dy\right.\\                                                                               
&\;+  \int_{\Ga_2\times \Om_2}\frac{ w(x,y)\psi(x)}{|x-y|^{n+2s}}+\left.\int_{\Ga_1\times \Ga_1}\frac{u(x,y)(\phi(x)- \phi(y))}{|x-y|^{n+2s}}dx dy+    
\int_{\Ga_2\times \Ga_2}\frac{w(x,y)(\psi(x)- \psi(y))}{|x-y|^{n+2s}}dxdy\right)\\                                                                                                                                                                                                                            
\leq&\e\left(\int_{Q}\frac{(u(x)-u(y))(\phi(x)- \phi(y))+(w(x)-w(y))(\psi(x)- \psi(y))}{|x-y|^{n+2s}} dxdy\right.\\
 &- \io(\la f(x) u_{+}^{-q}\phi+ \mu g(x) w_{+}^{-q}\psi) dx \left.- \frac{\al}{\al+\ba}\io b(x) u_{+}^{\al-1}w_{+}^{\ba}\phi 
 dx-\frac{\ba}{\al+\ba}\io b(x) u_{+}^{\al}w_{+}^{\ba-1}\phi\right)\\                                                         
&+2\e\left(\int_{\Ga_1\times \Om_1}\frac{|u(x)-u(y)|^{2}}{|x-y|^{n+2s}}dx dy\right)^{\frac{1}{2}}\left(\int_{\Ga_1\times      
\Om_1}\frac{|\phi(x)|^{2}}{|x-y|^{n+2s}}dx dy\right)^{\frac{1}{2}}\\                                                          
&+2\e\left(\int_{\Ga_2\times \Om_2}\frac{|w(x)-w(y)|^{2}}{|x-y|^{n+2s}}dx dy\right)^{\frac{1}{2}}\left(\int_{\Ga_2\times      
\Om_2}\frac{|\psi(x)|^{2}}{|x-y|^{n+2s}}dx dy\right)^{\frac{1}{2}}\\                                                          
&+2\e\left[\left(\int_{\Ga_1\times \Om^c}\frac{|u(x)|^{2}}{|x-y|^{n+2s}}dxdy\right)^{\frac{1}{2}}\left(\int_{\Ga_1\times      
\Om^c}\frac{|\phi(x)|^{2}}{|x-y|^{n+2s}}dx dy\right)^{\frac{1}{2}}\right.\\                                                   
&+2\e\left[\left(\int_{\Ga_2\times \Om^c}\frac{|w(x)|^{2}}{|x-y|^{n+2s}}dxdy\right)^{\frac{1}{2}}\left(\int_{\Ga_2\times      
\Om^c}\frac{|\psi(x)|^{2}}{|x-y|^{n+2s}}dx dy\right)^{\frac{1}{2}}\right.\\                                                   
 &+ \left(\int_{\Ga_1\times \Om_1}\frac{|u(x)-u(y)|^{2}}{|x-y|^{n+2s}}dx dy\right)^{\frac{1}{2}}\left(\int_{\Ga_1\times \Om_1}\frac{|\phi(x)|^{2}}{|x-y|^{n+2s}}dx dy\right)^{\frac{1}{2}}\\
&+ \left(\int_{\Ga_2\times \Ga_1}\frac{|w(x)-w(y)|^{2}}{|x-y|^{n+2s}}dx dy\right)^{\frac{1}{2}}\left(\int_{\Ga_2\times \Ga_1}\frac{|\psi(x)|^{2}}{|x-y|^{n+2s}}dx dy\right)^{\frac{1}{2}}\\
&\quad+\left.\left({\int_{\Ga_1\times \Ga_1}}\frac{|u(x)-u(y)|^{2}}{|x-y|^{n+2s}}dx dy\right)^{\frac{1}{2}}\left(\int_{\Ga_1\times \Ga_1}\frac{|\phi(x)-\phi(y)|^{2}}{|x-y|^{n+2s}}dx dy\right)^{\frac{1}{2}}\right]\\
&\quad+\left.\left({\int_{\Ga_2\times \Ga_2}}\frac{|w(x)-w(y)|^{2}}{|x-y|^{n+2s}}dx dy\right)^{\frac{1}{2}}\left(\int_{\Ga_2\times \Ga_2}\frac{|\psi(x)-\psi(y)|^{2}}{|x-y|^{n+2s}}dx dy\right)^{\frac{1}{2}}\right]\\
&+ \e \e^{\al} \|b\|_{\infty} \left(\int_{\Ga_1}|\phi|^{\al+\ba}dx\right)^{\frac{\ba}{\al+\ba}}\left(\int_{\Ga_1}(w_+)^{\al+\ba}dx\right)^{\frac{\ba}{\al+\ba}}\\
&+\e \e^{\ba} \|b\|_{\infty}\left(\int_{\Ga_2}(u_+)^{\al+\ba}dx\right)^{\frac{\al}{\al+\ba}} \left(\int_{\Ga_2}|\phi|^{\al+\ba}dx\right)^{\frac{\ba}{\al+\ba}}\\
&-\frac{\e\al}{\al+\ba} \int_{\Ga_1} b(x) u_{+}^{\al-1}w_{+}^{\ba}\phi dx-\frac{\e\ba}{\al+\ba} \int_{\Ga_2} b(x)u_{+}^{\al} w_{+}^{\ba-1}\psi dx.
\end{align*}}
Since the measure of $\Ga_i =\{x\in \Om| (u_i +\e \phi_i)(x)\leq 0\}$ tend to zero as $\e\ra 0$, it follows that
\[\int_{\Ga_i\times \Om_i}\frac{|\phi_i(x)|^{2}}{|x-y|^{n+2s}}dxdy\ra 0,\; \mbox{as}\; \e\ra 0,\] and similarly
\[\int_{\Ga_i\times \Om^c}\frac{|\phi_i(x)|^{2}}{|x-y|^{n+2s}}dx dy, \int_{\Ga_i\times \Ga_i}\frac{|\phi_i(x)-\phi_i(y)|^{2}}{|x-y|^{n+2s}}dx dy,\] $\int_{\Ga_1} b(x) u_{+}^{\al-1}w_{+}^{\ba}\phi dx$ and $\int_{\Ga_2} b(x)u_{+}^{\al} w_{+}^{\ba-1}\psi dx$,  all are tend to $0$ as $\e\ra 0$. Dividing by $\e$ and letting $\e\ra 0$, we obtain
{\small \begin{align*}
\int_{Q}& \frac{(u(x)-u(y)) (\phi(x)-\phi(y))+(w(x)-w(y)) (\psi(x)-\psi(y))}{|x-y|^{n+2s}}dxdy\\
 &- \io (\la f(x) u_{+}^{-q}\phi+\mu g(x) w_{+}^{-q}\psi) dx - \frac{\al}{\al+\ba} \io b(x) u_{+}^{\al-1}w_{+}^{\ba}\phi dx - \frac{\ba}{\al+\ba} \io b(x) u_{+}^{\al}w_{+}^{\ba-1}\psi dx \geq 0
\end{align*}}
and since this holds equally well for $(-\phi,-\psi)$, it follows that $(u,w)$ is indeed a positive weak solution of problem $(P_{\la,\mu}^{+})$
and hence a positive solution of $(P_{\la,\mu})$.\QED
\end{proof}
\begin{Lemma}\label{lm1}
 There exists a minimizing sequence $\{(U_k,W_k)\}$ in $\mc N_{\la,\mu}^{-}$ such that $(U_k,W_k)\ra (U,W)$ strongly in $\mc N_{\la,\mu}^{-}$.
Moreover $(U,W)$ is a positive weak solution of $(P_{\la,\mu})$.
\end{Lemma}
\begin{proof}
Using the Ekeland variational principle again, we may find a minimizing sequence $\{(U_k,W_k)\}\subset \mc N_{\la,\mu}^{-}$ for the minimizing problem
 $\inf_{\mc N_{\la,\mu}^{-}}J_{\la,\mu}$ such that for $(U_k,W_k)\rightharpoonup (U,W)$ weakly in $\h$ and pointwise a.e. in $Q$. We can repeat the argument used in Lemma \ref{le6} to derive that when $(\la,\mu)\in \Ga)$
\begin{align}\label{s11}
(1+q)\io (\la f(x) U_{+}^{1-q}+ \mu g(x)W_{+}^{1-q})dx - (\al+\ba-2)\io b(x)U_{+}^{\al} W_{+}^{\ba} dx<0
\end{align}
which yields
\[(1+q)\io (\la f(x) (U_{k})_{+}^{1-q}+ \mu g(x)(W_{k})_{+}^{1-q})dx -(\al+\ba-2)\io b(x) (U_{k})_{+}^{\al}(W_k)_{+}^{\ba}dx\leq -C_4\]
for $k$ sufficiently large and a suitable positive constant $C_4$. At this point we may proceed exactly as in Lemmas \ref{le7}, \ref{le8}, \ref{le9}, \ref{le10} and corollary \ref{cc1}, we conclude that $U$, $W>0$ is the required positive weak solution of problem $(P_{\la,\mu}^{+})$.
In particular $(U,W)\in \mc N_{\la,\mu}$. Moreover from \eqref{s11} it follows that
\begin{align*}
&(1+q)\|(U,W)\|^2 -(\al+\ba-1+q)\io b(x)U_{+}^{\al} W_{+}^{\al} dx\\
=&(1+q)\left[ K_{\la,\mu}(U,W) + \io b(x)U_{+}^{\al} W_{+}^{\ba} dx\right]- (\al+\ba-1+q)\io b(x)U_{+}^{\al} W_{+}^{\ba} dx\\
=&(1+q)K_{\la,\mu}(U,W) -(\al+\ba-2)\io b(x) U_{+}^{\al}W_{+}^{\ba} dx<0,
\end{align*}
that is $(U,W)\in\mc N_{\la,\mu}^{-}$.\QED
\end{proof}

\noi {\bf Proof of the Theorem \ref{th1}:} From Lemmas \ref{le10}, \ref{lm1} and \ref{le2}, we can conclude that the problem $(P_{\la,\mu})$ has at least two positive weak solutions $\ee\in\mc N_{\la,\mu}^{+}$, $(U,W)\in\mc N_{\la,\mu}^{-}$ with $\|(U,W)\|>\|\ee\|$ for any $(\la,\mu)\in\Ga$. \QED




\noindent{\bf Acknowledgements:} The author's research is supported
by National Board for Higher Mathematics, Govt. of India, grant
number: 2/40(2)/2015/R$\&$D-II/5488.


\end{document}